\documentclass{amsart}
\usepackage{euscript,graphicx,epstopdf,amscd,amsgen,amsfonts,amssymb,latexsym,amsmath,amsthm,graphicx,mathrsfs,times,color,overpic,bbm}

\usepackage{xcolor}
\definecolor{forestgreen}{rgb}{0.0, 0.27, 0.13}

\newtheorem{theorem}{Theorem}[section]
\newtheorem{lemma}[theorem]{Lemma}
\newtheorem{corollary}[theorem]{Corollary}
\newtheorem{proposition}[theorem]{Proposition}
\newtheorem{claim}[theorem]{Claim}

\theoremstyle{definition}
\newtheorem{example}[theorem]{Example}
\newtheorem{remark}[theorem]{Remark}

\def\bN{\mathbb{N}}

\def\bR{\mathbb{R}}
\def\bE{\mathbb{E}}

\def\bS{\mathbb{S}}
\def\bP{\mathbb{P}}

\def\cF{\mathcal{F}}

\def\bp{\mathbf{p}}

\DeclareMathOperator{\arc}{Arc}
\DeclareMathOperator{\Lip}{Lip}
\newcommand{\eqdef}{\stackrel{\scriptscriptstyle\rm def}{=}}

\numberwithin{equation}{section}

\DeclareMathSymbol{\varnothing}{\mathord}{AMSb}{"3F}
\title[Contracting on average IFS]{Contracting on average iterated function systems\\ by metric change}

\author[K. Gelfert]{Katrin Gelfert}
\address{Instituto de Matem\'atica, Universidade Federal do Rio de Janeiro, Cidade Universit\'aria - Ilha do Fund\~ao, Rio de Janeiro 21945-909,  Brazil}
\email{gelfert@im.ufrj.br}

\author[G. Salcedo]{Graccyela R. Salcedo}
\address{Instituto de Matem\'atica, Universidade Federal do Rio de Janeiro, Cidade Universit\'aria - Ilha do Fund\~ao, Rio de Janeiro 21945-909,  Brazil}
\email{salcedopg@im.ufrj.br}

\begin{document}

\begin{abstract}
We study contraction conditions for an iterated function system of continuous maps on a metric space which are chosen randomly, identically and independently. We investigate metric changes, preserving the topological structure of the space, which turn the IFS into one which is contracting on average. For the particular case of a system of $C^1$-diffeomorphisms of the circle which is proximal and does not have a probability measure simultaneously invariant by every map, we derive a strongly equivalent metric which contracts on average. 
\end{abstract}

\begin{thanks}
{This study was financed in part by the Coordena{\c c}{\~a}o de Aperfei{\c c}oamento de Pessoal de N{\`i}vel Superior (CAPES) -- Finance Code 001 and also partially by CNPq grants (Brazil).}
\end{thanks}

\subjclass[2000]{}

\maketitle

\section{Introduction}

Given a complete metric space $(M,d)$, an \emph{iterated function system (IFS)} is a finite set  $\mathcal{F} = \{f_0,\ldots,f_{N-1}\}$, $N\ge2$, of continuous maps $f_i\colon M\to M$, $i=0,\ldots,N-1$. One important goal is to understand the asymptotic behavior of consecutive concatenations of maps in $\cF$ where the choice at each step is according to probabilities $p_i$, $i=0,\ldots,N-1$, of some probability vector $\bp=(p_0,\ldots,p_{N-1})$, defining the triple $(\cF,\bp,d)$. This behavior is very well understood under the hypothesis that every map contracts uniformly. However, this is a rather strong requirement. The first rigorous treatment of contracting IFSs was done in the seminal work \cite{HUTCH}, which established the existence and uniqueness of the attractor and the stationary measure.
Several weaker hypotheses imply also good stochastic properties of the associated Markov chain generated by the IFS. For example, an IFS which is contracting on average (the concatenated maps do not necessarily shrink the distance between two points at every step and everywhere, but in expectation they do; see  \eqref{def:CA}) has a contracting (hence unique) stationary Borel probability measure. 

Though, much less is known for an IFS of maps which either ``just do not expand'' or have simultaneously some  ``contracting regions'' and  some ``expanding regions'' or even repelling fixed points (compare the examples depicted in Figure \ref{fig.3}). In such a general context, to gain any average contraction, one least topological requirement was coined in \cite{SteDav2} is that ``the orbit of a point wanders sufficiently around the space to pick up an average contraction''. Such property would, for example, call for an IFS which is minimal%
\footnote{The IFS $\cF$ is \emph{forward minimal} if for every nonempty closed set $A\subset M$ satisfying $f_i(A)\subset A$ for every $i=0,\ldots,N-1$ it holds $A=M$. The IFS $\cF$ is \emph{backward minimal} if $\cF^{-1}=\{f_i^{-1}\}$ is forward minimal.}. 

Observe that the existence and uniqueness of a Markov chain-stationary measure does not depend on the metric (as long as metrics generate the same topology), while contraction properties do.  A natural question is when for $(\cF,\bp,d)$ there exists some metric $D$ on $M$ equivalent to $d$ such that $(\cF,\bp,D)$ is, for example, contracting on average. Besides surveys such as \cite{Kai:81,DF,Ios:09}, we point out \cite{sten} which provides an ample discussion of many kinds of contracting conditions  and \cite{LesSniStr:20} which reviews IFSs from a more topological point of view, both mentioning also the method of metric change. In Section \ref{sec:impact} we briefly discuss some stochastic properties that remain true for the contracting on average IFS after a metric change. 

One step towards this direction was done in \cite{KatSten} where a convenient metric change  turns a backward minimal IFS of homeomorphisms on $\bS^1$ into a non-expansive on average one (see also \cite{MR4036718}). 
Here, for a non-expansive on average IFS of Lipschitz maps on a compact metric space, we give (sufficient and necessary) conditions to guarantee a metric change which turns an IFS into a contracting on average one. Moreover, we thoroughly discuss several local and global contraction-like properties intimately related with it.

The existence of a metric which makes a \emph{dynamical system}  ``contracting'' or ``expanding'' has been widely studied. Let us cite some key results.  If a map ``eventually'' contracts (expands) in the sense that its $k$th iteration has this property, a convenient change of the metric turns it into a contraction (expansion)  in its first iteration (see, for example, \cite[Chapter 4]{PFUM}). As explained, for example, in \cite [Chapter 4]{Shu:87} a  hyperbolic structure of a diffeomorphism is independent of the Riemannian metric on the ambient manifold. In \cite{FD2,Fat:89}, using  Frink's metrization theorem, it is shown that for any expansive homeomorphism of a compact topological space there exists an equivalent metric such that the map contracts (expands) on stable (unstable) sets. 
Analogously, there exists a metric which turns a (positively) expansive continuous map of a compact metric space into an expanding one (see \cite[Chapter 4]{PFUM}).%
\footnote{The existence of so-called adapted metrics in partially hyperbolic dynamics and dynamics with a dominated splitting was investigated in \cite{Gou:07,HiPuSh}. The use of adapted norms and metrics is also common in the study of nonuniform hyperbolicity when analyzing the size of local un-/stable manifolds  (see, for example, \cite{BarPes:07} in the $C^{1+\varepsilon}$ case and \cite[Section 8]{AbdBonCro:11} for a $C^1$ dominated setting).}

Before stating our main results, let us introduce the main contraction properties which we are going to investigate. 
A precursor assumed in  \cite{DF} requires ``contraction in mean'': $f_i$ are Lipschitz  with Lipschitz constants $\Lip(f_i)$, $i=0,\ldots,N-1$, satisfying
\[
	\sum_{i=0}^{N-1}p_i\Lip(f_i)
	<1.
\] 
Weaker concepts are proposed in \cite{BDEG88,PeignM93}, where $(\cF,\bp,d)$ is assumed to \emph{contract on average (CA)} in the sense that there exists $\lambda\in(0,1)$ so that
\begin{equation}\label{def:CA}
	\sum_{i=0}^{N-1}p_id(f_i(x),f_i(y))
	\le \lambda d(x,y)
	\quad\text{ for every } x,y\in M.
\end{equation}
 In fact, \cite{BDEG88} requires even weaker assumptions allowing for place-dependent probabilities and for contraction in $L^q$ norm, $q>0$, while  \cite{PeignM93} assumes ``contraction on average after some iterations'' (that is, it is $k$-eventually contracting in average, for some $k\in\bN$, in the sense defined in \eqref{eq01} below). 
Less is known about an IFS if we put only the weaker hypothesis that $(\cF,\bp,d)$ is \emph{non-expansive on average (NEA)} in the sense that
\[
	\sum_{i=0}^{N-1}p_id(f_i(x),f_i(y))
	\le  d(x,y)
	\quad\text{ for every } x,y\in M.
\]
This  property implies, for example, that the associated Markov chain is non-expansive (see \cite{Sza:03} and references therein).  
There are variations of these definitions such as being \emph{eventually strongly contracting on average (ESCA)}, \emph{synchronizing on average (SA)}, \emph{locally eventually contractive on average (LECA)}, and \emph{proximal}, that we postpone to Section \ref{sec:syncont}. To simplify the exposition, we will mainly use these short notations.

Recall that metrics $d$ and $D$ on some common space $M$ are \emph{(topologically) equivalent} if they generate the same topology. They are \emph{strongly equivalent} if there exist positive constants $a$ and $b$ such that $ad(x,y)\leq D(x,y)\leq bd(x,y)$ for every $x,y \in M$. Clearly, strong equivalence implies topological one, but not vice versa. Given $\alpha\in(0,1)$, note that $d^\alpha\colon M\times M\to[0,\infty)$ defined by $d^\alpha(x,y)\eqdef (d(x,y))^\alpha$ is a metric on $M$, and $d$ and $d^\alpha$ are equivalent.

\begin{figure}[h] 
 \begin{overpic}[scale=.5]{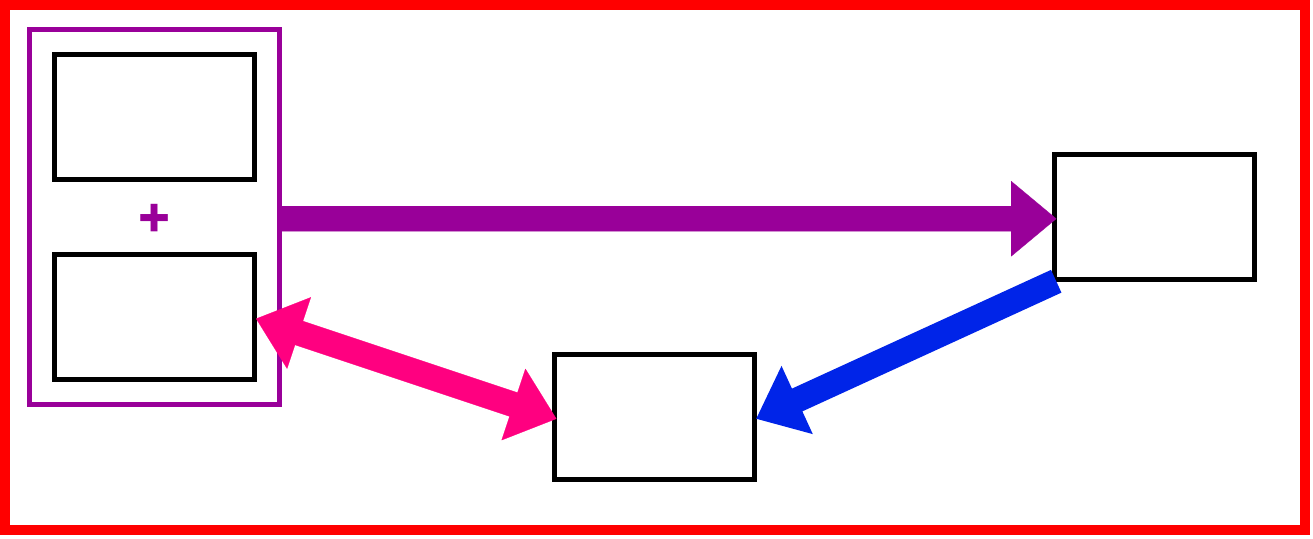}
	\put(54,35){{\textcolor{red}{NEA ({\tiny{assuming compactness}})}}}
	\put(5,30){{ESCA}}
	\put(84,22.2){{CA}}
 	\put(5,14.5){{LECA}}
	\put(46,7.5){{SA}}
	\put(40,27){\scriptsize{\textcolor{magenta}{(Theorem \ref{pro:1})}}}
	\put(20,5){\scriptsize{\textcolor{red}{(Lemma \ref{lem:2})}}}
	\put(69,11){\scriptsize{\textcolor{blue}{for all }}}
	\put(69,8){\scriptsize{\textcolor{blue}{equivalent metrics }}}
	\put(69,5){\scriptsize{\textcolor{blue}{ (Proposition \ref{pro:2})}}}
 \end{overpic}
  \caption{Some implications, assuming non-expansive on average (NEA) on a compact metric space: eventually strongly contracting on average (ESCA), contracting on average (CA), synchronizing on average (SA), locally eventually contractive on average (LECA).}
 \label{figure.4}
\end{figure}
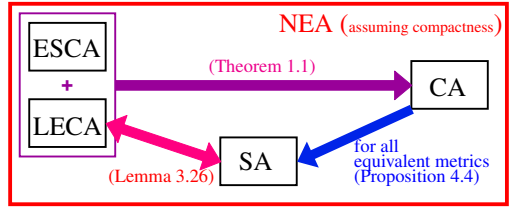

Let us now state our main results. Let $\Sigma^+_N\eqdef \{0,\ldots,N-1\}^\bN$ be the space of one-sided sequences and denote by $\mu$ the Bernoulli measure on $\Sigma^+_N$ determined by $\bp$.
For any sequence $\xi=(\xi_1 \xi_2 \ldots)\in\Sigma^+_N$, $n\geq 1$, and  $x\in M$ let
\[    f^n_{\xi}(x)
    \eqdef f_{\xi_1\ldots\xi_n}
    \eqdef f_{\xi_n}\circ f_{\xi_{n-1}}\circ\cdots\circ{f_{\xi_1}}(x), 
    \quad  
    f^0_{\xi}(x)\eqdef x.
\]Given $x,y\in M$ and $n\in\bN$, let
\[    Z_{n,d}^{x,y}(\xi)
    \eqdef d(f^n_{\xi}(x),f^n_{\xi}(y)),
    \quad
    Z_{0,d}^{x,y}(\xi)
    \eqdef d(x,y).
\]Given $\lambda>0$ and $n\in\bN$, consider the metric defined by
\[
 d_{n,\lambda}(x,y)
    \eqdef d(x,y)
    +\frac{1}{\lambda^{1/n}}\bE (Z_{1,d}^{x,y})
    +\cdots+\frac{1}{\lambda^{(n-1)/n}}\bE (Z_{n-1,d}^{x,y}),
\]
where  $\bE(\cdot)$ denotes the expected value according to the probability distribution.

\begin{theorem}\label{pro:1} 
Consider a compact metric space $(M,d)$ and a triple $(\cF,\bp,d)$, where $\bp$ is a non-degenerate probability vector and $\cF$ is an IFS of Lipschitz maps that is non-expansive on average. 
If $(\cF,\bp,d)$ satisfies LECA and ESCA, then there exist $\lambda\in(0,1)$ and $n\in\mathbb{N}$ such that $(\cF,\bp,d_{n,\lambda})$ is contracting on average. Moreover, LECA is equivalent to SA.
\end{theorem}

The second main result concerns the particular case of an IFS of $C^1$ diffeomorphisms of the circle $\bS^1$ equipped with the usual metric $d(x,y)\eqdef\min\{\lvert x-y\rvert,1-\lvert x-y\rvert\}$. 

\begin{theorem}\label{teo3.1}
Assume that $\mathcal{F}$ is an IFS of $C^1$-diffeomorphisms on $\bS^1$.
Assume that $(\cF,d)$ is proximal and there does not exist a probability measure which is invariant by every map in $\mathcal{F}$. Then for every non-degenerate probability vector $\bp$ there exist $\alpha\in(0,1]$, $\lambda\in(0,1)$, and $n\in\mathbb{N}$ such  that $(\cF,\bp,D)$, with $D\eqdef(d^{\alpha})_{n,\lambda}$, is contracting on average. Moreover, $d\le D\le Cd^\alpha$ for some $C>0$, and hence $d$ and $D$ are strongly equivalent.
\end{theorem}

In Section \ref{sec:6} we illustrate and discuss our results (compare Figure \ref{fig.3}) in two classes of homeomorphisms on $\bS^1$. We summarize their main properties.
\begin{figure}[h] 
 \begin{overpic}[scale=.6]{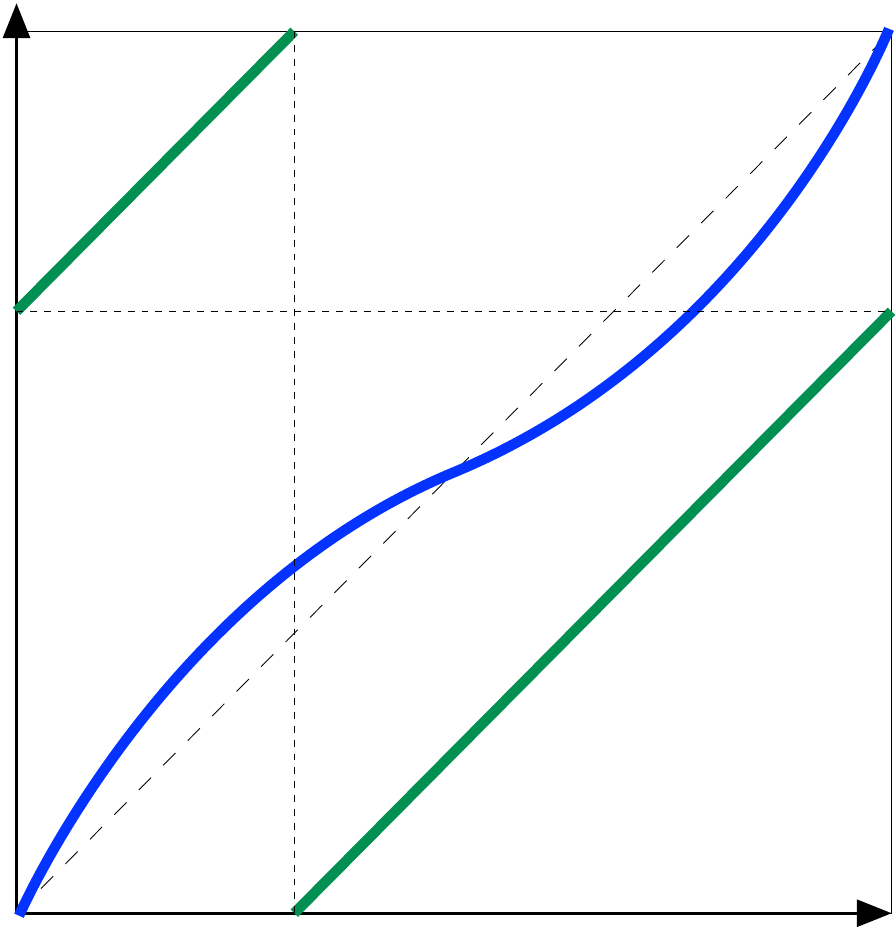}
	\put(97,67){$\textcolor{forestgreen}{f_0}$}
	\put(97,93){$\textcolor{blue}{f_1}$}
 \end{overpic}
 \hspace{0.7cm}
  \begin{overpic}[scale=.6]{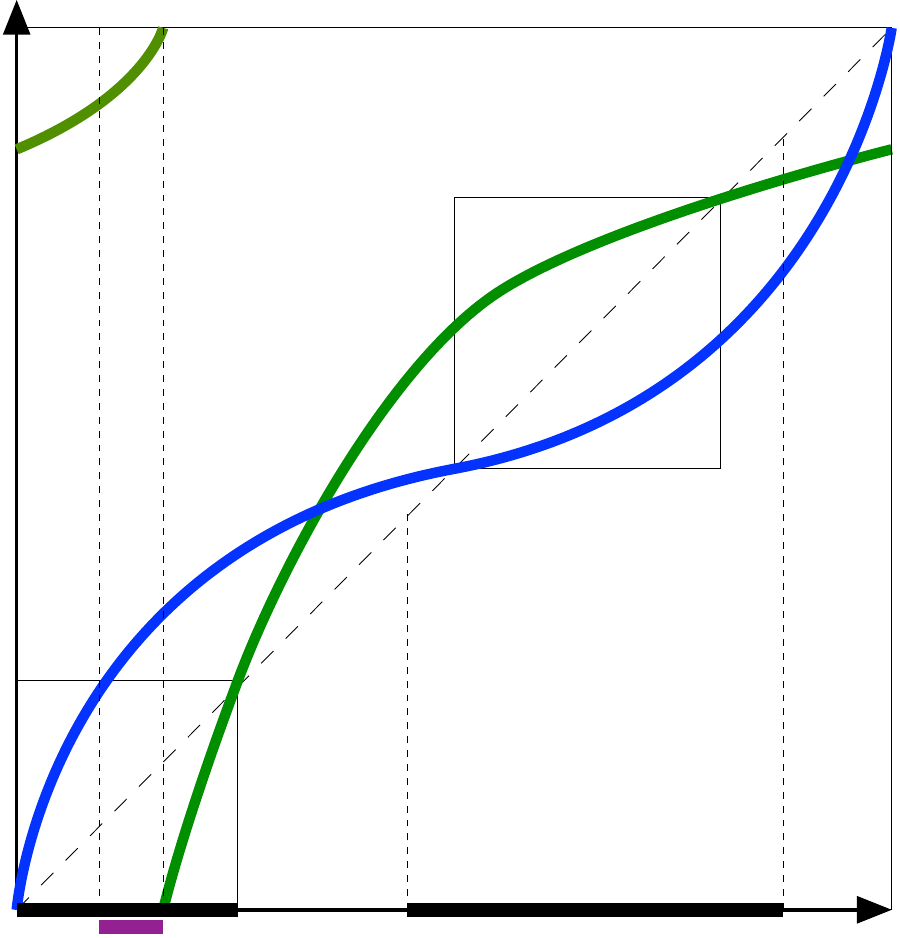}
	\put(61,-4){$I$}
	\put(1,-4){$J$}
	\put(11,-4){$\textcolor{magenta}{J^\ast}$}
	\put(97,83){$\textcolor{forestgreen}{f_0}$}
	\put(97,93){$\textcolor{blue}{f_1}$}
 \end{overpic}
  \caption{Examples studied in Section \ref{sec:exe1} (left) and Section \ref{sec:exe2} (right)}
 \label{fig.3}
\end{figure}
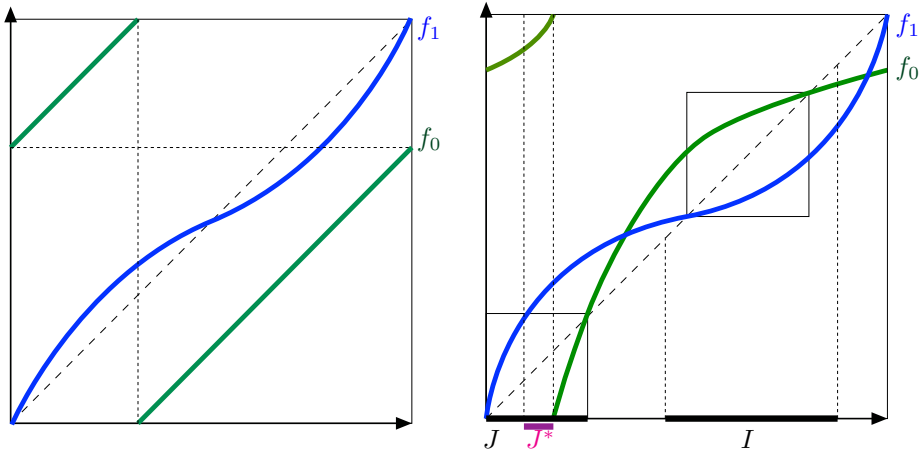

\begin{example}[$(\cF,\bp,d)$ in Section \ref{sec:exe1}, see Figure \ref{fig.3} (left)]
This example of an IFS of $C^1$-diffeomorphisms which is proximal (and hence SA and LECA), but fails to be NEA.  The choice of metric $\rho$ in \cite{KatSten} forces $(\cF,\bp,\rho)$ to satisfy NEA, SA, and LECA, but $(\cF,\bp,\rho)$ fails to be ESCA. As $(\cF,\bp,\rho)$ verifies the hypotheses of Theorem \ref{teo3.1}, there  exist $\alpha\in(0,1]$, $\lambda\in(0,1)$ and $n\in\mathbb{N}$ such  that for 
$
	D\eqdef(d^{\alpha})_{n,\lambda}
$	 
the triple $(\cF,\bp,D)$ is CA. In particular, $(\cF,\bp,D)$ satisfies  NEA, SA, LECA, and ESCA. However,  for all $\beta\in(0,1]$ the metric $D^{\beta}$ fails to be strongly equivalent to $d$. 
\end{example}

\begin{example}[$(\cF,\bp,d)$ in Section \ref{sec:exe2}, see Figure \ref{fig.3} (right)]
In this example the approach in \cite{KatSten} does not apply. This example fails to be NEA and $\varepsilon$-LCA, but it  is proximal, S, SA, and LECA. For this example, it is shown that for appropriate $\alpha\in(0,1]$ the metric
\[
	\hat D(x,y)
	\eqdef \bE\Big( \sup_{n\geq 0}Z_{n,d^{\alpha}}^{x,y}\Big),
\]
is strongly equivalent to $d^{\alpha}$ and $(\cF,\bp,\hat D)$ satisfies NEA, SA (and hence LECA), and ESCA. Hence, by Theorem \ref{pro:1}, there is a metric $ D$ which is strongly equivalent to $\hat D$ (and hence to $d^{\alpha}$) such that $(\cF,\bp, D)$ is CA. In particular, $(\cF,\bp, D)$ satisfies  NEA, LECA, and ESCA. Moreover, if $f_0$ and $f_1$ are $C^1$-diffeomorphisms  which have no common fixed points, then Theorem \ref{teo3.1} applies.
\end{example}

\begin{figure}[h] 
 \begin{overpic}[scale=.6]{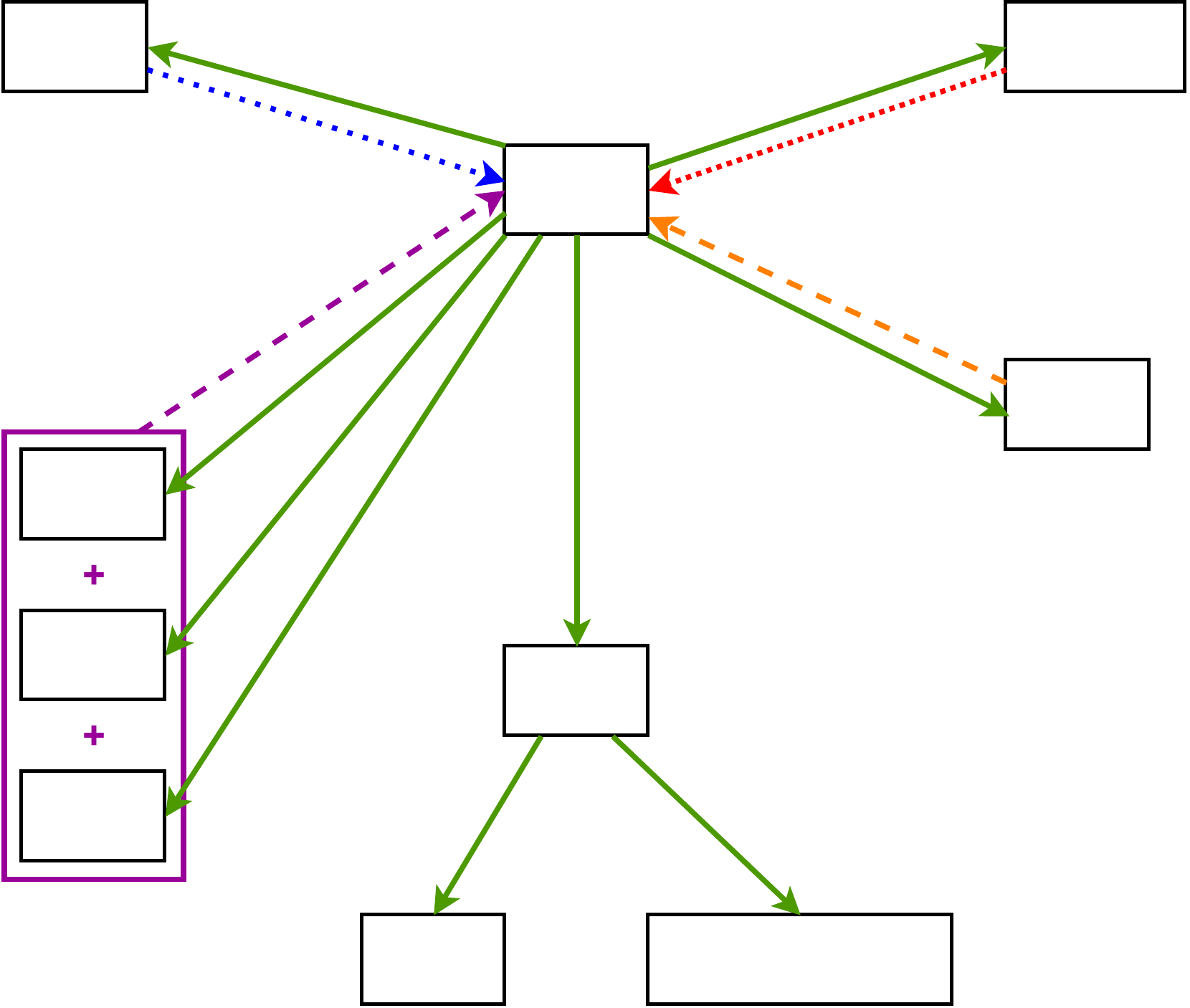}
	\put(56,3){{$\varepsilon$-local $\log$-CA}}
	\put(31,3){{LCWS}}
	\put(46,67.2){{CA}}
 	\put(1.5,80){{$k$-ECA}}
 	\put(4,42){{NEA}}
 	\put(3,28){{LECA}}
 	\put(3,14.8){{ESCA}}
 	\put(87,79){{$\log$-CA}}
 	\put(86,49){{SA${}_{\rm exp}$}}
	\put(43.5,25){{$\varepsilon$-LCA}}
	\put(-10,72){\scriptsize{\textcolor{blue}{for some equivalent }}}
	\put(-10,69){\scriptsize{\textcolor{blue}{metric (Proposition \ref{prolem1:1})}}}
	\put(-10,61){\scriptsize{\textcolor{violet}{for some strongly-equivalent metric}}}
	\put(-10,58){\scriptsize{\textcolor{violet}{assuming compactness}}}
	\put(-10,55){\scriptsize{\textcolor{violet}{(Theorem \ref{pro:1})}}}
	\put(67.5,70.5){\scriptsize{\textcolor{red}{for some $d^{\alpha}$ and $\alpha\in(0,1]$}}}
	\put(67.5,67.5){\scriptsize{\textcolor{red}{(Proposition \ref{prologca2})}}}
	\put(67.5,64){\scriptsize{\textcolor{orange}{for some equivalent metric}}}
	\put(67.5,61){\scriptsize{\textcolor{orange}{(Proposition \ref{pro:3})}}}
	\put(62.5,49){\scriptsize{\textcolor{forestgreen}{for bounded metric}}}
 \end{overpic}
  \caption{Implications between: contracting on average (CA), eventually contracting on average (ECA), locally eventually contractive on average (LECA), non-expansive on average (NEA), eventually strongly contracting on average (ESCA), $\log$-contractive on average ($\log$-CA), $\varepsilon$-local contractive on average ($\varepsilon$-LCA), locally contractive in the weak sense (LCWS), and $\varepsilon$-local $\log$-CA}
 \label{fig.1}
\end{figure}
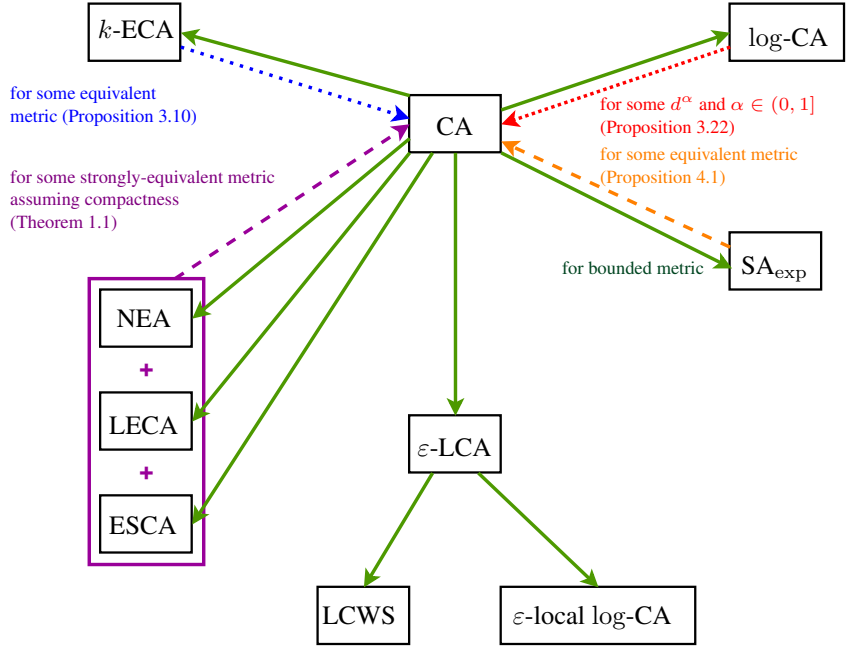

This paper is organized as follows. In Section \ref{sec:impact} we briefly discuss what impact a metric change has on stochastic properties for a contracting on average IFS.
In Section \ref{sec:syncont}, we define all concepts used in this paper and explore some of their properties and prove the implications depicted  in Figure \ref{fig.1}.
In the context of NEA, Figure \ref{figure.4}  depicts implications that will be explored in the course of this paper. Figure \ref{fig.1} puts the property CA into relation with other ones discussed here. In Section \ref{sec:CondCA}, we present some sufficient conditions that guarantee the existence of some metric for which the system is CA and, in particular, prove Theorem \ref{pro:1}. In Section \ref{sec:5}, we study IFSs of $\bS^1$ and prove Theorem \ref{teo3.1}. The above summarized examples are studied in Section \ref{sec:6}.

\section{Stochastic properties after metric change}\label{sec:impact}

In this section, we briefly describe stochastic properties for contracting on average IFSs and discuss the impact of a metric change. Recall that a Borel probability measure $\nu$ on $M$ is \emph{stationary} for the IFS $(\cF,\bp)$ if
\[
	\cF_\ast\nu=\nu,
	\quad\text{ where }\quad
	\cF_\ast\nu
	\eqdef \sum_{i=0}^{N-1}p_i(f_i)_\ast\nu.
\]
This term is justified by the fact that if $(i_n)_{n\in\bN}$ is a stochastic sequence with values in $\{0,\ldots,N-1\}$ which is independently and equally distributed according to the probability vector $\bp$ and  $x\in M$ is a $\nu$-distributed random variable, independent of $(i_n)_{n\in\bN}$, then $(W_n^x)_{n\in\bN}$, defined by
\[	W_n^x\eqdef (f_{\xi_n}\circ\cdots\circ f_{\xi_1})(x),
	\quad
	W_0^x\eqdef x,
\] is a stationary stochastic sequence. The independent random applications of maps $f_i$ from the IFS $\cF$, each chosen with probability $p_i$, gives rise to the Markov chain $(W_n^x)_{n\in\bN}$.
 
By our hypotheses, this Markov chain has the \emph{weak Feller property}%
\footnote{The \emph{transfer operator} $T$ associated to the pair $(\cF,\bp)$ acts on the space of bounded measurable functions $\varphi\colon M\to\bR$ by 
\[
	T\varphi(x)
	\eqdef \sum_{i=0}^{N-1}p_ih(f_i(x)).
\]
It has the \emph{weak Feller property} if it maps the space of real valued continuous functions on $M$ to itself.}
and hence there exist at least one stationary measure. It is relevant to know under what conditions such stationary measure is unique, under what hypotheses it is true that
\[
	P^n (x ,B)
	\eqdef \bP\{W_n^x\in B\}
	\to \nu(B)
\]
as $n\to\infty$, and what is the speed of such a convergence.

For a compact metric space $(M, D)$ and $(\cF,\bp,D)$ which is CA,  by \cite[Theorem 2.1]{BDEG88} there exists a unique stationary Borel probability measure $\nu$. By \cite[Corollary 2.1]{JarnerTweedie},  for any initial conditions $x$, the distribution of $W_n^x$ converges exponentially fast to $\nu$ in the Prokhorov metric, that is, for every measurable set $B$ and $n\in\mathbb{N}$  it holds
\[
	P^n(x,B)
	\leq \nu(B_n)+A_xr^n,\quad
	\nu(B)
	\leq P^n(x,B_n)+A_xr^n,
\]	
where $B_n\eqdef\{y\in M\colon D(y,B)<A_xr^n\}$. Here the rate of convergence $r\in(0,1)$ does neither depend on $n$ nor on $x$. Furthermore, the constant $A_x$ does not depend on $n$ and is uniformly bounded.
Note that all previous facts do not depend on the metric on $M$ (within the class of  metrics which generate the same topology). Only the explicit convergence in the Prokhorov metric was given in terms of $D$. Observe that if $D$ and another metric $d$ are such that $C^{-1}d^\alpha\le D\le Cd^\alpha$ for some constants $C^{-1},\alpha\in(0,1]$ (that is, $D$ and $d^\alpha$ are strongly equivalent), then we still obtain exponential contraction taking 
\[
	B_n
	\eqdef \{y\in M\colon d(y,B)<C_xr^{n/\alpha}\}, 
	\quad\text{where}\quad 
	C_x
	\eqdef(CA_x)^{1/\alpha}.
\]	

Assuming CA, in  \cite[page 484]{Elton87} an Ergodic Theorem was shown, whose  assertion is unaltered under any metric change. Assuming the slightly more general property $k$-ECA for any $k\in\bN$, in \cite[Theorem 5.1]{PeignM93} a strong law of large numbers and a central limit theorem are stated; again these assertions remain the same under metric change.

\section{Synchronization and contraction (on average)}\label{sec:syncont}

In this section we discuss several types of  synchronization-like and contraction conditions and their relations between each other.
Unless stated otherwise, we always assume that $(M,d)$ is a general metric space, $\mathcal{F}=\{f_0,\ldots,f_{N-1}\}$ an IFS of continuous maps, and $\bp$ a non-degenerate probability vector. Let $\mu$ the Bernoulli measure on $\Sigma^+_N$ determined by $\bp$.
When $X$ is a random variable on $(\Sigma^+_N,\mu)$, we write 
\[\bE (X)\eqdef \int_{\Sigma^+_N} Xd\mu.\]

\subsection{Synchronization}

The study of synchronization effects goes back to, at least, the 17th century, when
Huygens \cite{HC} observed the synchronization of linked pendulums. In the theory of dynamical systems, \emph{synchronization} usually refers to the phenomenon that for any two initially fixed distinct points their randomly chosen trajectories converge to each other. Let us now recall related concepts.

Let us briefly recall some properties. One says that $(\cF, \bp,d)$ is \emph{synchronizing (S)} if random orbits of different initial points converge to each other with probability 1, that is, for every $x,y\in M$ and almost every $\xi\in\Sigma^+_N$ it holds
\begin{align}\label{sinc}
    \lim_{n\to\infty}Z_{n,d}^{x,y}(\xi)=0.
\end{align}
The triple $(\cF, \bp,d)$ is \emph{exponentially synchronizing (S${}_{\rm exp}$)} if the convergence in \eqref{sinc} is exponentially fast, that is, if for every $x,y\in M$ and almost every $\xi\in\Sigma^+_N$ there exist $\lambda\in(0,1)$ and $C>0$ such that 
\[
	Z_{n,d}^{x,y}(\xi)\leq C \lambda^n.
\]	
The pair $(\cF,d)$ is \emph{proximal} if for every $x, y \in M$, there exist $\xi\in\Sigma^+_N$ and an increasing sequence $(n_k)_{k\in\mathbb{N}}$ such that
\[
	\lim_{k\to\infty}Z_{n_k,d}^{x,y}(\xi)=0.
\]	
The triple $(\cF, \bp,d)$ is called \emph{synchronizing on average (SA)}%
\footnote{Note that  the definition in \cite{MEMI,GAKV} differs from the one given here: $(\cF,\bp,d)$ is \emph{synchronizing on average} if for every $x,y\in M$, for almost every $\xi\in\Sigma_N^{+}$ it holds
\[\lim_{n\to\infty}\frac{1}{n}\sum_{j=0}^{n-1}Z_{j,d}^{x,y}(\xi)=0.\]}
if  for every $x,y\in M$ it holds
\[
    \lim_{n\to\infty}\bE \big(Z_{n,d}^{x,y}\big)=0.
\]
The triple $(\cF, \bp,d)$ is called \emph{exponentially synchronizing on average} (SA${}_{\rm exp}$) if there exist  $\lambda\in(0,1)$ and $C>0$ such that for every $x,y\in M$ it holds
\[ \bE \big(Z_{n,d}^{x,y}\big)\leq C \lambda^n.
\]
The following general relations hold between the above defined properties. The first one is an immediate consequence of the dominated convergence theorem. To prove the second one, we also use Jensen's inequality.

\begin{lemma}\label{sec3lem1}
	Assuming that $(M,d)$ is bounded, S implies SA.
\end{lemma}

\begin{lemma}\label{sec3lem2}
S and proximal are invariant under any change of equivalent metrics.
SA${}_{\rm exp}$ is invariant under any change of metrics $d$ and $D$ such that $d^{\alpha}$,  for some $\alpha \in(0,1]$, is strongly equivalent to $D$.
\end{lemma}

\begin{lemma}\label{SAimpliesprox}
	The following implications hold
\[
	\text{S}{}_{\rm exp}
	\quad\Rightarrow\quad \text{S}
	\quad\Rightarrow \quad\text{SA}
	\quad\Rightarrow \quad\text{proximal}.
\]	
\end{lemma}

\begin{proof}
The first two implications are immediate.
Let us assume $(\cF, \bp,d)$ is SA. Fix $x,y\in M$. Then $Z_{n,d}^{x,y}$ converges to $0$ in $L^1$ as $n\to\infty$. Hence, applying Chebyshev’s inequality, for every $\varepsilon>0$ it holds
\[
	\mu(Z_{n,d}^{x,y}\geq \varepsilon)
	\leq \varepsilon^{-1}\bE (Z_{n,d}^{x,y})
	\to0
\]
as $n\to\infty$, that is, $Z_{n,d}^{x,y}$ converges to $0$ in probability. By \cite[Theorem 2.3.2]{Durrett}, there exists a sub-sequence $(n_k)_k$ such that $Z_{n_k,d}^{x,y}$ converges almost surely to 0 as $k\to \infty$. This implies proximality.
\end{proof}

\begin{remark}
If $\mathcal{F}=\{f_i\}_{i=0}^{N-1}$ is an IFS of homeomorphisms of the circle $\bS^1$ which do note have a common fixed point and $\bp$ a non-degenerate probability vector,  then by \cite[Theorem E]{Malicet} the concepts S${}_{\rm exp}$, S,  and proximal (and hence SA) are equivalent. 
\end{remark}

\subsection{Global average contraction conditions}\label{sec:22}

The IFS $\cF$ is \emph{contracting (C)} if every map in $\mathcal{F}$ is a contraction. Given $k\in\bN$, the IFS $\cF$ is \emph{$k$-eventually contracting ($k$-EC)}
 if $\cF^k$ is contracting, where $\cF^k\eqdef\{f_{\xi_1\ldots\xi_k}\colon (\xi_1,\ldots,\xi_k)\in\{0,\ldots,N-1\}^k\}$, and \emph{eventually contracting (EC)} if it is $k$-eventually contracting for some $k\in\bN$.

\subsubsection{CA}
The triple $(\cF, \bp,d)$ is called \textit{contractive
on average (CA)} if there is some \emph{contraction rate} $\lambda\in(0,1)$ such that 
\[    \bE (Z_{1,d}^{x,y})\leq \lambda d(x,y)
    \quad\text{for every $x,y\in M$}.
\]The following property is straightforward to check.

\begin{remark}
CA is a particular case of the first hypothesis of \cite[Theorem 2.1]{BDEG88}.
\end{remark}

\begin{lemma}\label{lemalambda}
If $(\cF,\bp,d)$ is CA with contraction rate $\lambda\in(0,1)$, then for every $n\in\mathbb{N}$ 
\[
	\bE (Z^{x,y}_{n,d})\leq \lambda^n d(x,y)
    \quad\text{for every $x,y\in M$}.
\]
\end{lemma}

\begin{remark}
If $(M,d)$ is bounded, then it follows from Lemma \ref{lemalambda} that CA implies SA${}_{\rm exp}$.
\end{remark}

Given $k\in\bN$, the triple $(\cF, \bp,d)$ is called \emph{$k$-eventually contractive
on average ($k$-ECA)} if there exist some \emph{contraction rate} $\lambda\in(0,1)$ such that 
\begin{align}\label{eq01}
    \bE (Z_{k,d}^{x,y})\leq \lambda d(x,y)
        \quad\text{for every $x,y\in M$}.
\end{align}
The triple $(\cF,\bp,d)$ is \emph{eventually contracting on average (ECA)} if it is $k$-ECA for some $k\in\bN$.

The following lemma justifies that a, perhaps  obvious, first choice of a metric to establish CA is well defined.

\begin{lemma}\label{lemaABC1}
For every $n\in\bN$, the function $\psi_n\colon M\times M\to \bR$ defined by 
\[
	\psi_n(x,y) \eqdef\bE (Z_{n,d}^{x,y})
\]	 
is continuous and defines a pseudometric on $M$. For every $x,y\in M$, it holds
\[
	\sum_{i=0}^{N-1} p_i\psi_n(f_i(x),f_i(y))=\psi_{n+1}(x,y).
\]	
Moreover, if $\cF$ is an IFS of homeomorphisms then $\psi_n$ is a metric on $M$  which is equivalent to $d$. 
\end{lemma}
\begin{proof}
Continuity is immediate. 
The property $\psi_n(x,y)=\psi_n(y,x)$ is immediate by the symmetry of the metric $d$. As
\[\begin{split}
	\psi_n(x,z)
	= \sum_{i=0}^{N-1}p_id(f_\xi^n(x),f_\xi^n(z))
	&\le \sum_{i=0}^{N-1}p_i\big(d(f_\xi^n(x),f_\xi^n(y))+d(f_\xi^n(y),f_\xi^n(z))\big)\\
	&=\bE(Z_{n,d}^{x,y})+\bE(Z_{n,d}^{y,z}),
\end{split}\]
the triangle inequality holds true for $\psi_n$.    Thus, $\psi_n$ is a pseudo metric on $M$. The second property is immediate.

Finally note that $\psi_n(x,y)=0$ if and only if, $d(f_{\xi_1\ldots\xi_n}(x),f_{\xi_1\ldots\xi_n}(y))=0$ for all $\xi_1,\ldots,\xi_n\in\{0,\ldots,N-1\}$. Hence, if every $ f_i $ is a homeomorphism, then $x=y$. 
In this case, it also is immediate to see that $\psi_n$ generates the same topology as $ d $.
\end{proof}

Given $k\in\bN$ and $\lambda\in(0,1)$, consider $d_{k,\lambda}\colon M\times M\to [0,\infty)$ defined by
\begin{equation}\label{metricaECA}
    d_{k,\lambda}(x,y)
    \eqdef d(x,y)
    +\frac{1}{\lambda^{1/k}}\bE (Z_{1,d}^{x,y})
    +\cdots+\frac{1}{\lambda^{(k-1)/k}}\bE (Z_{k-1,d}^{x,y}).
\end{equation}
The following is an immediate consequence of Lemma \ref{lemaABC1}.

\begin{lemma}\label{lemrem:strequ}
 The function $d_{k,\lambda}$ defined in \eqref{metricaECA} is a metric which is  equivalent to $d$. Moreover, if all the maps in $\cF$ are Lipschitz, then in $d_{k,\lambda}$ is strongly equivalent to $d$.
\end{lemma}

\begin{proposition}\label{prolem1:1}
If $(\cF, \bp,d)$ is $k$-ECA with contraction rate $\lambda\in(0,1)$, then $(\cF, \bp,d_{k,\lambda})$ is CA with contraction rate $\lambda^{1/k}$.
\end{proposition}

\begin{proof}
Assume that $(\cF, \bp,d)$ is ECA with contraction rate $\lambda\in(0,1)$ and $k\in\bN$.
If $k=1$, then $d_{k,\lambda}=d$ and $(\cF, \bp,d)$ is CA. 
If $k\ge2$, then it follows from the definition of $d_{k,\lambda}$ that
\begin{align*}
    \mathbb{E}
&\Big(Z_{1,d_{k,\lambda}}^{x,y}\Big)
	=\sum_{i=0}^{N-1}p_id_{k,\lambda}(f_i(x),f_i(y))\\
&= \sum_{i=0}^{N-1}p_i\Big(d (f_i(x),f_i(y))
	+ \frac{1}{\lambda^{1/k}}\mathbb{E}\Big(Z_{1,d}^{f_i(x),f_i(y)}\Big)
	+\cdots
	+ \frac{1}{\lambda^{(k-1)/k}}\mathbb{E}\Big(Z_{k-1,d}^{f_i(x),f_i(y)}\Big)\Big)\\
&=\mathbb{E}\Big(Z_{1,d}^{x,y}\Big)
	+\frac{1}{\lambda^{1/k}}\mathbb{E}\Big(Z_{2,d}^{x,y}\Big)
	+\cdots+\frac{1}{\lambda^{(k-2)/k}}\mathbb{E}\Big(Z_{k-1,d}^{x,y}\Big)
	+\frac{1}{\lambda^{(k-1)/k}}\mathbb{E}\Big(Z_{k,d}^{x,y}\Big)\\
&\le\mathbb{E}\Big(Z_{1,d}^{x,y}\Big)
	+\frac{1}{\lambda^{1/k}}\mathbb{E}\Big(Z_{2,d}^{x,y}\Big)
	+\cdots+\frac{1}{\lambda^{(k-2)/k}}\mathbb{E}\Big(Z_{k-1,d}^{x,y}\Big)
	+\frac{1}{\lambda^{(k-1)/k}}\lambda d(x,y)\\
&=\lambda^{1/k}\Big(d(x,y)
	+\frac{1}{\lambda^{1/k}}\mathbb{E}\Big(Z_{1,d}^{x,y}\Big)
	+\frac{1}{\lambda^{2/k}}\mathbb{E}\Big(Z_{2,d}^{x,y}\Big)
	+\cdots+\frac{1}{\lambda^{(k-1)/k}}\mathbb{E}\Big(Z_{k-1,d}^{x,y}\Big)
	 \Big)\\
&=\lambda^{1/k} d_{k,\lambda}(x,y).
\end{align*}
Hence, $(\cF,\bp,d_{k,\lambda})$ is CA with contraction rate $\lambda^{1/k}$.
\end{proof}

\subsubsection{NEA}

The triple $(\cF, \bp,d)$ is said to be  \textit{non-expansive on average (NEA)} if 
\begin{align}\label{eq2}
    \bE (Z_{1,d}^{x,y})\leq d(x,y)
        \quad\text{for every $x,y\in M$}.
\end{align}

\begin{remark}
The NEA property was introduced in \cite{JarnerTweedie} as \emph{non-separating on average}. See also \cite[Part II]{Sza:03} for a study of NEA iterated function systems and associated non-expansive Markov operators. 
\end{remark}

The following is an immediate consequence of Jensen's inequality.

\begin{lemma}
If $(\cF, \bp,d)$ is NEA then for any $\alpha\in(0,1)$ the triple $(\cF, \bp,d^{\alpha})$ is NEA.
\end{lemma}

\begin{lemma}\label{lemrem:SA}
Assume $(M,d)$ is bounded. If $(\cF, \bp,d)$ is SA and $D$ is a metric  equivalent to $d$ such that $(\cF, \bp,D)$ is NEA, then $(\cF, \bp,D)$ is SA.
\end{lemma}

\begin{proof}
 If $(\cF, \bp,d)$ is SA, then $\bE \big(Z_{n,d}^{x,y}\big)\to0$ as $n\to\infty$ for every $x,y\in M$. Arguing as in the proof of Lemma \ref{SAimpliesprox}, there exists a sub-sequence $(n_k)_k$ such that $Z_{n_k,d}^{x,y}\to0$ almost surely to 0 as $k\to \infty$. Since $D$ and $d$ are assumed to be  equivalent, $Z_{n_k,D}^{x,y}\to$ almost surely as $k\to \infty$. Hence, by the Dominated Convergence Theorem, $\bE (Z_{n_k,D}^{x,y})\to0$ as $k\to \infty$. Since $(\cF, \bp,D)$ is NEA, $\bE \big(Z_{n,D}^{x,y}\big)\to0$ as $n\to\infty$. This proves that $(\cF, \bp,D)$ is SA.
\end{proof}

We state the following straightforward implications without proof.

\begin{lemma}
	C implies CA,  CA implies NEA, and EC implies ECA.
\end{lemma}

\begin{remark}
	If $(\cF, \bp,d)$ is NEA, then for all $k\in\bN$ and $\lambda\in(0,1)$ the metric $d_{k,\lambda}$ defined in \eqref{metricaECA} is strongly equivalent to $d$. 
\end{remark}

\begin{lemma}\label{lemaABC}
For $n\geq 0$, let $\psi_n(x,y)\eqdef\bE (Z_{n,d}^{x,y})$ be as in Lemma \ref{lemaABC1}. If $(\cF,\bp,d)$ is NEA, then for every $x,y\in M$ it holds
\begin{enumerate}
    \item  $\psi_1(x,y)\leq d(x,y)$,
    \item $\psi_n(x,y)$ is non-decreasing in $n$ and hence the limit $\lim_{n\to\infty}\psi_n(x,y)$ exists,
    \item Assuming that $M$ is compact, if for every $x,y\in M$ we have
$\lim_{n\to\infty}\psi_n(x,y)=0$, then $\psi_n\to 0$ uniformly.
\end{enumerate}
\end{lemma}

In the following proof and below we use the following simplifying notation  
\[
	 p_{\xi_1\ldots\xi_n}\eqdef p_{\xi_1}\cdots p_{\xi_n}.
\]

\begin{proof}
Property (1) just restates the definition of NEA.
To show (2), check that
\[\begin{split}
        \psi_{n+1}(x,y)&=\bE (Z_{n+1,d}^{x,y})
        =\sum_{\xi_1,\ldots,\xi_{n+1}=0}^{N-1}
        	p_{\xi_1,\ldots,\xi_{n+1}}d(f_{\xi_1,\ldots,\xi_{n+1}}(x),f_{\xi_1,\ldots,\xi_{n+1}}(y))\\
        &=\sum_{\xi_1,\ldots,\xi_n=0}^{N-1}p_{\xi_1,\ldots,\xi_n}\sum_{\ell=1}^Np_\ell\,
        		d\big(f_\ell(f_{\xi_1,\ldots,\xi_n}(x)),f_\ell(f_{\xi_1,\ldots,\xi_n}(y))\big)\\
	&=\sum_{\xi_1,\ldots,\xi_n=0}^{N-1}p_{\xi_1,\ldots,\xi_n}
		\bE\Big(Z^{f_{\xi_1,\ldots,\xi_n}(x),f_{\xi_1,\ldots,\xi_n}(y)}_{1,d}\Big)\\
	{\tiny{\text{(by NEA)}}}\quad	
        &\leq\sum_{\xi_1,\ldots,\xi_n=0}^{N-1}
        		p_{\xi_1,\ldots,\xi_n}d(f_{\xi_1,\ldots,\xi_n}(x),f_{\xi_1,\ldots,\xi_n}(y))
        =\bE (Z_{n,d}^{x,y})
        =\psi_n(x,y).
\end{split}\]
This, together with $\psi_n\ge0$ implies item (2). 

To prove (3), assume that $M$ is compact and $\psi_n\to0$ point-wise. As the limit function is continuous, by Dini's theorem, convergence is uniform. 
\end{proof}

\begin{remark}
By \cite[Proposition 1]{KatSten} any IFS $\mathcal{F}=\{f_i\}_{i=0}^{N-1}$ of homeomorphisms of $(\bS^1,d)$ which is backward minimal there exists a metric $\rho$ on $\bS^1$  equivalent to $d$ on $\bS^1$ such that $(\cF, \bp,\rho)$ is NEA. 
\end{remark}

\subsubsection{$\log$-CA}

The triple $(\cF,\bp,d)$  is called \emph{$\log$-contractive on average ($\log$-CA)}  if there exists $\lambda < 1$ such that 
\[    \prod_{j=0}^{N-1} d(f_{j}(x),f_{j}(y))^{p_{j}}\leq \lambda d(x,y)
            \quad\text{for every $x,y\in M$},
\]or, equivalently, 
\[
	\bE \left(\log\frac{Z_{1,d}^{x,y}}{d(x,y)} \right)\leq \log \lambda<0
	        \quad\text{for every $x,y\in M, x\ne y$}.
\]

\begin{remark}
Consider $\mathcal{F}$ be an IFS of Lipschitz maps. Let $L$ defined by \eqref{Lconst:Lip}.
In \cite{DF}, the condition $\mathbb{E}(\log L)<0$, was called \emph{contracting on average}
which is stronger than $\log$-CA. The condition $\log$-CA was introduced in \cite[Page 84]{Elton87}.
\end{remark}

\begin{lemma}\label{logca1}
	CA implies $\log$-CA.
\end{lemma}

\begin{proof}
Assuming CA with contraction rate $\lambda\in(0,1)$, note that $\bE (Z_{1,d}^{x,y}/d(x,y))\leq \lambda$. Hence, by Jensen's inequality
$$\bE \left(\log\frac{Z_{1,d}^{x,y}}{d(x,y)} \right)\leq \log \bE \left(\frac{Z_{1,d}^{x,y}}{d(x,y)} \right)\leq \log\lambda<0,$$
proving the lemma.
\end{proof}

\begin{remark}
The concept of $\log$-CA was introduced  in \cite{BE}. 
Assuming $(M, d)$ to be a complete metric space, $\mathcal{F}$ to be an IFS of Lipschitz maps, and $(\cF, \bp, d)$ to be $\log$-CA, they prove the existence of an attractive (hence unique) stationary measure (extending previous results obtained in the case when $M$ is compact, see references in \cite{BE}).
\end{remark}

The following example is presented in \cite{Edalat} to illustrate that $\log$-CA is weaker than C. Indeed, it also shows that  $\log$-CA is weaker than CA.

\begin{example}[$\log$-CA, but not NEA and not CA]
Let $M=[0,1]$ and $d(x,y)\eqdef |x-y|$. Let $\bp=(\frac{1}{2},\frac{1}{2})$. Consider the IFS $\mathcal{F}=\{f_0,f_1\}$ given by
\[
	f_0,f_1\colon M\to M,\quad
	f_0(x)\eqdef\frac{x}{3},\,
	f_1(x)\eqdef\min\{1,2x\}.
\]
Note that for all $x, y \in [0, 1]$ it holds
\[
	\left[d(f_0(x),f_0(y))\right]^{1/2}\left[d(f_1(x),f_1(y))\right]^{1/2}
	\leq \frac{2}{3}|x-y|
\]	
and hence $(\cF,\bp,d)$ is $\log$-CA. On the other hand, for $x,y\in[0,1/2]$ it holds  $f_1(x)=2x$ and $f_1(y)=2y$, so that
\begin{align*}
    \bE (Z_{1,d}^{x,y})
    &=\frac{1}{2}\left(d(f_0(x),f_0(y))+d(f_1(x),f_1(y)) \right)
    =\frac{1}{2}\left(\frac{1}{3}d(x,y)+2d(x,y) \right),
\end{align*}
which implies that $(\cF,\bp,d)$ is not NEA and thus not CA.
\end{example}

 \begin{proposition}\label{prologca2}
Assume that $\mathcal{F}$ is an IFS of Lipschitz maps and $(\cF, \bp, d)$ is $\log$-CA. Then, there exists $\alpha\in(0,1]$ such that $(\cF, \bp, d^{\alpha})$ is CA.
 \end{proposition}

 \begin{proof}
Consider 
\begin{equation}\label{Lconst:Lip}
     	L(\xi)=L(\xi_1)
	\eqdef\sup_{x,y\in M\colon x\neq y }
	\frac{d(f_{\xi_1}(x),f_{\xi_1}(y)}{d(x,y)},
\end{equation}
which is a Lipschitz constant for $f_{\xi_1}$. As, by hypothesis, $(\cF,\bp,d)$ is $\log$-CA with some contraction rate $\lambda\in(0,1)$, it holds 
\[
	\sup_{x\ne y}\bE\left(\log\frac{Z_{1,d}^{x,y}}{d(x,y)}\right)
	= \bE(\log L)
	\le \log\lambda
	<0. 
\]
Hence, by \cite[Lemma 2]{WW}, there exists $\alpha\in(0,1]$ such that $\bE (L^{\alpha})\in(0,1)$. Note that $d^{\alpha}$ is a metric on $M$ and $L^{\alpha}(\xi)$ is a Lipschitz constant for $f_{\xi_1}$ relative to $(M,d^{\alpha})$. It follows that for all $x,y\in M$ 
\[    \bE (d^{\alpha}(X_1^x,X_1^y))\leq \bE (L^{\alpha}) d^{\alpha}(x,y),
\]
and hence $(\cF, \bp, d^{\alpha})$ is CA with contraction rate $\bE (L^{\alpha})\in(0,1)$.
\end{proof}

\begin{remark}
It is common in the literature to assume a log-CA condition instead of CA, see for example \cite{sten}. In the present setting, by Lemma \ref{logca1} and  Proposition \ref{prologca2}, these conditions are equivalent (changing $d$ by $d^{\alpha}$ if necessary, for some $\alpha\in(0,1]$).
\end{remark}

\subsection{Local average contraction conditions: LECA and ESCA}\label{sec:23}

In this section, we discuss several types of local average contraction conditions for IFSs. In particular, we introduce LECA and ESCA, which are, besides NEA, key properties towards Theorem \ref{pro:1}. Such conditions have been studied, for example, in \cite{TK} and later, independently, in \cite{SteDav,SteDav2} and  \cite{Carlsson}. See also  \cite{JarnerTweedie,LagSten}. Here our focus is on conditions which are sufficient or necessary for CA, possibly after some change of metric.

The first property is a generalization of CA. We say that $(\cF,\bp,d)$ is \emph{locally eventually contractive on average (LECA)} if  
\begin{equation}\label{ip}
   \text{for every $x,y\in  M$ there exist $\ell\geq 1$ and $\lambda\in(0,1)$:}\quad
   \bE (Z_{\ell,d}^{x,y})\leq \lambda d(x,y).  
\end{equation}

\begin{remark}
If $(\cF,\bp,d)$ is LECA then for any $\alpha\in(0,1)$ the triple $(\cF,\bp,d^{\alpha})$ is LECA.
\end{remark}

\begin{lemma}\label{LECAimpl}
	CA implies LECA.
	SA implies LECA.
\end{lemma}

Lemma \ref{LECAimpl} is an immediate consequence. It shows that synchronization (on average) is intimately related with contraction (on average). The next result indeed proves that, assuming NEA, those properties are equivalent. For its proof we need to recall some more concepts.  The set $\Sigma_{N}^{+}$ is naturally equipped with the product topology on $\{0,\ldots, N-1\}^{\mathbb{N}}$, where $\{0,\ldots, N-1\}$ is given the discrete topology. A basis is given by the family of \emph{cylinders}
\[
	[i;\xi_1,\ldots,\xi_n]
	\eqdef \{\eta\in\Sigma_{N}^{+}\colon \eta_{i+1}=\xi_1,\ldots,\eta_{i+n}=\xi_n\}.
\]	
We simply write $[\xi_1,\ldots,\xi_n]\eqdef[1;\xi_1,\ldots,\xi_n]$.	
Every cylinder is clopen. Every open set in $\Sigma_{N}^{+}$ is a countable union of cylinders. 

\begin{lemma}\label{lem:2}
Suppose that $(\cF, \bp,d)$ is  NEA on some compact metric space $(M,d)$. Then, $(\cF, \bp,d)$ is SA if and only if it is LECA.
\end{lemma}

\begin{proof}
By Lemma \ref{LECAimpl}, SA implies LECA. To prove the reverse implication, let us assume that $(\cF, \bp,d)$ is LECA. By Lemma \ref{lemaABC} (2), for every $x,y\in M$ the limit 
\[
   	\delta(x,y)
	\eqdef \lim_{n\to\infty}\psi_n(x,y)
	\ge0,
	\quad\text{ where }\quad
	\psi_n(x,y) \eqdef\bE (Z_{n,d}^{x,y}),
\]
exists. Arguing by contradiction, let us suppose that $(\cF, \bp,d)$ is not SA and hence there exist $x,y\in M$ such that $\delta=\delta(x,y)>0$. Define 
\[
	B_{\delta}\eqdef \left\{(z,w)\colon d(z,w)\geq\frac{\delta}{2}\right\},
\]	
which is a closed subset of $ M \times M $ and hence compact. Note that $d$ and $\psi_n $ are continuous on $M\times M$. Thus, for $n\geq 1$ and $t\in(0,1)$ the set
\[
	C_{n,t}
	\eqdef \{(z,w)\colon \bE (Z_{n,d}^{z,w})<t d(z,w)\}
\]	
is open in $M\times M$ and so 
\[
	A_{n,t}\eqdef B_{\delta}\cap C_{n,t}
\]	 
is open in $B_{\delta}$. As we assume LECA, for each $(z,w)\in B_{\delta}$ there exist $\ell\geq 1$ and $\lambda'\in(0,1)$ such that $\bE (Z_{\ell,d}^{z,w})\leq \lambda' d(z,w)$ and hence, $(z,w)\in A_{\ell,t}$ for every $t\in(\lambda',1)$. In particular, it holds
\[
	B_{\delta}
	=\bigcup_{t\in(0,1)}\bigcup_{n\in\mathbb{N}}A_{n,t}.
\]	 
By compactness of $B_{\delta}$, there exist $n_1,\ldots,n_k\in\bN$ and $t_1,\ldots,t_k\in(0,1) $ such that
\[
	B_\delta = A_{n_1, t_1} \cup \cdots \cup A_{n_k, t_k}.
\]	
Define
\[
	N \eqdef  \max \{n_1,\ldots, n_k\}
	\quad\text{ and }\quad
	\lambda \eqdef \max \{t_1, \ldots, t_k\}.
\]	
If $(z, w)$ in $B_\delta$, then $(z, w)\in A_{n_i, t_i}$ for some $i\in\{1,\ldots,k\}$. This together with Lemma \ref{lemaABC} (2) implies that
\begin{equation}\label{alfa}
	\bE (Z_{N,d}^{z,w})
	\le \bE  (Z_{n_i,d}^{z, w}) <t_i d (z, w) 
	\leq \lambda d (z, w).
\end{equation}
Let
\[
	\Gamma_n
	\eqdef \left\{\xi\colon (f_{\xi}^n(x),f_{\xi}^n(y))\in B_{\delta}\right\}.
\]	
Note that $\Gamma_n$ is the union of cylinder sets. Indeed, if $\xi\in\Gamma_n$ then every $\eta\in[\xi_1,\ldots,\xi_n]\in\Gamma_n$.
For $n\geq 0$, we have
\begin{align}\label{eq0lem2}
    \bE (Z_{N+n,d}^{x,y})&=\bE \left(Z_{N+n,d}^{x,y}\mathbbm{1}_{\Gamma_n}\right)+\bE \left(Z_{N+n,d}^{x,y}\mathbbm{1}_{\Gamma_n^c}\right).
\end{align}
For the first term, we observe
\begin{align*}
    \bE &\left(Z_{N+n,d}^{x,y}\mathbbm{1}_{\Gamma_n}\right)\\
    &=\sum_{\xi_{n+1},\ldots, \xi_{N+n}}
    	\sum_{[\xi_1\ldots\xi_n]\subset\Gamma_n}
	p_{\xi_{n+1}\ldots\xi_{N+n}}
	p_{\xi_1\ldots\xi_n}
	d(f_{\xi_{1}\ldots\xi_{N+n}}(x),f_{\xi_1\ldots\xi_{N+n}}(y))\\
    &=\sum_{[\xi_1,\ldots,\xi_n]\subset\Gamma_n}p_{\xi_1\ldots\xi_n}
    	\sum_{\xi_{n+1},\ldots, \xi_{N+n}}
	p_{\xi_{n+1}\ldots\xi_{N+n}}
	d(f_{\xi_{1}\ldots\xi_{N+n}}(x),f_{\xi_{n+1}\ldots\xi_{N+n}}(y))\\
    &=\sum_{[\xi_1,\ldots,\xi_n]\subset\Gamma_n}
    	p_{\xi_1\ldots\xi_n}\bE \left(Z_{N,d}^{f_{\xi_1\ldots\xi_n}(x),f_{\xi_1\ldots\xi_n}(y)} \right).
\end{align*}
By definition of $\Gamma_n$ it holds $(f_{\xi_1\ldots\xi_n}(x),f_{\xi_1\ldots\xi_n}(y))\in B_{\delta}$, it follows from \eqref{alfa} that
\begin{align}\label{eq1lem2}\begin{split}
    	\bE \left(Z_{N+n,d}^{x,y}\mathbbm{1}_{\Gamma_n}\right)
    	&\leq\sum_{[\xi_1,\ldots,\xi_n]\subset\Gamma_n}p_{\xi_1\ldots\xi_n} 
		\lambda \,d\left(f_{\xi_1\ldots\xi_n}(x),f_{\xi_1\ldots\xi_n}(y) \right)\\
	&= \lambda\, \bE (Z_{n,d}^{x,y}\mathbbm{1}_{\Gamma_n}).
\end{split}\end{align}
For the second term, we have
\begin{align*}
    	\bE \left(Z_{N+n,d}^{x,y}\mathbbm{1}_{\Gamma^c_n}\right)
  	=\sum_{[\xi_1,\ldots,\xi_n]\subset\Gamma_n^c}
		p_{\xi_1\ldots\xi_n}
		\bE \left(Z_{N,d}^{f_{\xi_1\ldots\xi_n}(x),f_{\xi_1\ldots\xi_n}(y)} \right).
\end{align*}
Since $(\cF, \bp,d)$ is  NEA, again using Lemma \ref{lemaABC} (2), we get
\begin{align*}
    	\bE \left(Z_{N+n,d}^{x,y}\mathbbm{1}_{\Gamma^c_n}\right)
  	\leq \sum_{[\xi_1,\ldots,\xi_n]\subset\Gamma_n^c}
		p_{\xi_1\ldots\xi_n}
			d\left(f_{\xi_1\ldots\xi_n}(x),f_{\xi_1\ldots\xi_n}(y) \right).
\end{align*}
By definition of $\Gamma_n$ we have that $(f_{\xi_1\ldots\xi_n}(x),f_{\xi_1\ldots\xi_n}(y))\notin B_{\delta}$  so that
\begin{align}\label{eq2lem2}
    	\bE \left(Z_{N+n,d}^{x,y}\mathbbm{1}_{\Gamma^c_n}\right)
  	< \frac{\delta}{2}\mu(\Gamma_n^c).
\end{align}
From \eqref{eq0lem2}, \eqref{eq1lem2} and \eqref{eq2lem2} we get
\[
	 \bE (Z_{N+n,d}^{x,y})
	 < \lambda \,\bE \left(Z_{n,d}^{x,y}\mathbbm{1}_{\Gamma_n}\right)
	 	+\frac{\delta}{2}\mu\left( \Gamma_n^c\right).
\]		
Moreover, by induction on $k\in\bN$, it follows
\[\begin{split}
    \bE (Z_{kN,d}^{x,y})
    &=\bE \left(Z_{kN,d}^{x,y}\mathbbm{1}_{\Gamma_{(k-1)N}}\right)
    +\bE \left(Z_{kN,d}^{x,y}\mathbbm{1}_{\Gamma_{(k-1)N}^c}\right)\\
    &\leq  \lambda\, \bE \left(Z_{(k-1)N,d}^{x,y}\mathbbm{1}_{\Gamma_{(k-1)N}}\right)
    +\frac{\delta}{2}\mu\left( \Gamma_{(k-1)N}^c\right)\\
    &\leq  \lambda^2\, \bE \left(Z_{(k-2)N,d}^{x,y}\mathbbm{1}_{\Gamma_{(k-1)N}\cap\Gamma_{(k-2)N}}\right)+\frac{\delta}{2}\mu\left(\Gamma_{(k-1)N}^c\cup\Gamma^c_{(k-2)N}\right)\\
    \le\ldots
    &\leq  \lambda^k \,\bE \left(Z_{0,d}^{x,y}\mathbbm{1}_{\bigcap_{j=0}^{k-1}\Gamma_{jN}}\right)
    	+\frac{\delta}{2}\mu\Big(\bigcup_{j=0}^{k-1}\Gamma_{jN}^c\Big)\\
    &=\lambda^kd(x,y)\mu\Big(\bigcap_{j=0}^{k-1}\Gamma_{jN}\Big)
    	+\frac{\delta}{2}\mu\Big(\bigcup_{j=0}^{k-1}\Gamma_{jN}^c\Big).
\end{split}\]
Hence, recalling that $ \lambda\in(0,1)$ and using $\mu(\cdot)\le 1$, we get
\[
	\lim_{k\to\infty}\bE (Z_{kN,d}^{x,y})\leq \frac{\delta}{2},
\]	
which is a contradiction. 
This implies SA.
\end{proof}

The following is a consequence of Lemmas \ref{lem:2} and \ref{lemrem:SA}.

\begin{corollary}
Assume $(M,d)$ is compact.  Assume that $d$ and  $D$ are  equivalent metrics such that $(\cF, \bp,d)$ and $(\cF, \bp,D)$ are NEA. Then $(\cF, \bp,d)$ is LECA if and only if $(\cF, \bp,D)$ is LECA.
\end{corollary}

The following is a generalization of the definition of locally contractive with respect to the reverse system \cite[Definition 5]{SteDav} and of the definition $\varepsilon$-local (average) contractive \cite[Definition 1]{sten}. We say that $(\cF,\bp,d)$ is \emph{eventually strongly contracting on average (ESCA)} if for every  $x\in  M$ there exist $\ell\ge 1$ and an open neighborhood $V_{(x,x)}\subset M\times M$ of $(x,x)$ such that 
\begin{equation}\label{iip}
    \sup_{(y,z)\in V_{(x,x)}, y\neq z}
    \frac{\bE (Z_{\ell,d}^{z,y})}{ d(z,y)}<1.
\end{equation}

\begin{lemma}
If $(\cF,\bp,d)$ is ESCA then  for any $\alpha\in(0,1)$ the triple $(\cF,\bp,d^{\alpha})$ is ESCA.
\end{lemma}

\begin{remark} 
In Section \ref{sec:exe2}, the example $(\cF,\bp,d^{\alpha})$ given is ESCA with parameter $\ell\ne1$. Moreover, $(\cF,\bp,d^{\alpha})$ is LECA, but fails to be NEA.
\end{remark}

\begin{lemma}\label{Lem:1}
If $(\cF,\bp,d)$ is CA, then for every $\alpha\in(0,1]$ and a metric $D$ which is strongly equivalent to $d^{\alpha}$, for every $k\in\bN$ large enough the triple $(\cF,\bp,D)$ is $k$-ECA.
\end{lemma}

\begin{proof}
Suppose $(\cF,\bp,d)$ be CA with contracting rate $\lambda\in(0,1)$. Given $\alpha\in(0,1]$ and a metric $D$  on $M$ strongly  equivalent  to $d^\alpha$, take $b>a>0$ and  such that $aD\le d^{\alpha}\le bD$. Note that $(\cF,\bp,d^{\alpha})$ is CA with contracting rate $\lambda^{\alpha}\in(0,1)$. Take $k\in\mathbb{N}$ so that 
\[
	\frac{b}{a}\lambda^{\alpha k}<1.
\]	
First, strong equivalence implies that 
\[
	a \bE (Z_{k,D}^{x,y})
	\le \bE (Z_{k,d^{\alpha}}^{x,y})
                \quad\text{for every $x,y\in M$}.	
\]
For every $x,y\in M$, $x\neq y$, CA with contraction rate $\lambda^{\alpha}$ together with the above implies 
\[
	\lambda^{\alpha k}
	\geq \frac{\bE (Z_{k,d^{\alpha}}^{x,y})}{d^{\alpha}(x,y)}
	=\frac{\bE (Z_{k,d^{\alpha}}^{x,y})}{\bE (Z_{k,D}^{x,y})}\cdot
		\frac{\bE (Z_{k,D}^{x,y})}{D(x,y)}\cdot
		\frac{D(x,y)}{d^{\alpha}(x,y)}
	\ge a\cdot	\frac{\bE (Z_{k,D}^{x,y})}{D(x,y)} \cdot \frac1b.
\]	
Hence, it follows
\[
		\bE (Z_{k,D}^{x,y})
	\le \frac ba\lambda^{\alpha k}D(z,y),
\]		
which proves that $(\mathcal{F},\bp,D)$ is $k$-ECA with contraction rate $\frac ba\lambda^{\alpha k}\in(0,1)$.
\end{proof}

\subsection{Further contraction conditions}
In this section, we continue our discussion of contraction conditions and put them into the context of the ones defined above. Although, none of the concepts defined in this section will be implemented in the remainder of this paper.
Let us define for every $x\in M$ the sequence of random variables $\left(X_n^x\right)_{n\geq 0}$ on $(\Sigma^+_N,\mu)$ by
\[
    X_n^x(\xi):=f_{\xi_n}\circ f_{\xi_{n-1}}\circ\cdots\circ f_{\xi_1}(x); \quad X_0^x(\xi):=x.
\]
For $A\subset M$ measurable, denote by $\tau_A(x)\eqdef \inf\{n\geq 1\colon X_n^x\in A\}$ the first time the process $X_n^x$ hits the set $A$. Following \cite{JarnerTweedie}, $(\cF,\bp,d)$ satisfies the \emph{local contraction property relative to $A\subset M$} if there is $\lambda\in(0,1)$ such that
\begin{equation}\label{cont:3}
    \bE \Big(d\big(X^x_{\tau_A(x)\vee\tau_A(y))},
    			X^y_{\tau_A(x)\vee\tau_A(y))}\big)  \Big)
    \leq \lambda d(x,y)
                \quad\text{for every $x,y\in M$},
\end{equation}
that is, ``there is some contraction after the set $A\subset M$ is reached''. In other terms, this condition states that if we start two chains, at $x$ and $y$, respectively, and run them simultaneously using the same maps, then at the time both of them have visited $A$, in average they will be closer to each other by a factor $\lambda$. 
\begin{remark}
In \cite[Section 3]{JarnerTweedie}, it is shown that $\tau_A(x)\vee\tau_A(y)<\infty$ almost surely, assuming that there exists a  function
 $V\colon M\to [1,\infty)$, satisfying $\sup_{x\in A} V(x)<\infty$, and constants $r\in(0,1)$ and $b<\infty$ such that for every $x\in M$
\begin{align}\label{cont:4}
  \bE  V\left(X_1^x\right )
   \leq r V(x)+b\mathbbm{1}_A(x).
\end{align}
 Furthermore, in \cite{JarnerTweedie} the function $V$ is used to control the behavior outside of $A$. Moreover, assuming that $(\cF, \bp, d)$ is NEA and $M$ is complete separable metric space with bounded metric $d$, thus guaranteeing the existence of a unique stationary probability measure (see \cite[Theorem 2.1]{JarnerTweedie}). 
 \end{remark}
 
 \begin{lemma}
 	Assume that $(\cF,\bp,d)$ satisfies the local contraction property relative to $A=M$, then $(\cF,\bp,d)$ is CA and \eqref{cont:4} is satisfied for $V\equiv 1$.
\end{lemma}

The following example  $(\cF,\bp,d)$ satisfies the local contraction property \eqref{cont:3}, but is not CA. Furthermore, it is NEA and satisfies \eqref{cont:4}. 

\begin{example}[NEA, but not CA] 
Adapting an example in \cite[Section 6]{JarnerTweedie} to our context, let $M=[0,2]$ and $d(x,y)\eqdef\lvert x-y\rvert$.  Let $\bp=(p,1-p)$ for some $p\in(1/2,1)$. Consider the IFS $\cF=\{f_0,f_1\}$ given by
\[
	f_0,f_1\colon M\to M,\quad
	f_0(x)\eqdef
	\begin{cases}
			x-\frac{2}{3}	&x\ge 1,\\
			\frac{x}{3}		&x\le 1,
	\end{cases}
	\mbox{ and } 
	f_1(x)
	\eqdef \min\left\{x+\frac{2}{3},2\right\}.
\]	
It is easy to see that $(\cF,\bp,d)$ is NEA. On the other hand, for $x,y\in[1,4/3]$ it holds
\[
	\bE (Z_{1,d}^{x,y}))
	=\bE (d(X_1^x,X_1^y))=d(x,y)=|x-y|
\]
and hence $(\cF,\bp,d)$ is not CA.

Now let us show that $(\cF,\bp,d)$ satisfies the local contraction property \eqref{cont:3} relative to $A=[0,1]$. Since $f_0$ and $f_1$ are non-decreasing functions we have that for every $x,y\in M$, such that $x<y$, 
$$\tau_A(x)\leq \tau_A(y)$$ 
and if $y\in A$, then $x\in A$. Also, note that for $x,y\in A$ we have that
$$\bE (Z_{1,d}^{x,y})=\left(\frac{p}{3}+(1-p) \right)d(x,y).$$
Therefore, \eqref{cont:3} holds with $\lambda\eqdef \frac{p}{3}+(1-p)<1$. 

On the other hand, consider $f(t)=pe^{-\frac{2}{3}t}+(1-p)e^{\frac{2}{3}t}$. Note that $f(0)=1$ and $f'(0)<0$. Thus, to see that \eqref{cont:4} holds fix $t>0$ such that $f(t)<1$. Let $r\eqdef f(t)\in(0,1)$, $b=e^{\frac{5}{3}t}$ and $V(x)\eqdef e^{tx}$. If $x\in A$, then
$$\bE V(X_1^x)= pe^{t\frac{x}{3}}+(1-p)e^{t\left(x+\frac{2}{3}\right)},$$
and hence
$$\bE V(X_1^x)\leq pe^{\frac{t}{3}}+(1-p)e^{\frac{5}{3}t}\leq e^{\frac{5}{3}t}\leq r +b\leq r V(x)+b\mathbbm{1}_A(x).$$
If $x\in (1,2]=M\setminus A$, then
\[
	\bE V(X_1^x)
	\leq pe^{t\left(x-\frac{2}{3}\right)}+(1-p)e^{t\left(x+\frac{2}{3}\right)}
	=V(x)\left( pe^{-t\frac{2}{3}}+(1-p)e^{\frac{2}{3}t} \right)
	=r V(x)+\mathbbm{1}_A(x).
\]	
This proves \eqref{cont:4}.
\end{example}

The idea of locally contractive Markov chains can be expressed in several ways. Globally contracting on average-type CA and $\log$-CA are convenient because  they can be analysed by many different methods. In \cite{sten}, the following local average contraction conditions is considered. 
Given $\varepsilon>0$, the triple $(\cF,\bp,d)$ is called \emph{$\varepsilon$-local contractive on average ($\varepsilon$-LCA)} if there exists $\lambda\in(0,1)$ such that
\begin{align}\label{sten1}
    \sup_{0<d(x,y)<\varepsilon}\frac{\bE (Z_{1,d}^{x,y})}{d(x,y)}\leq \lambda.
\end{align}
The triple $(\cF,\bp,d)$ is called \emph{$\varepsilon$-local $\log$-CA} if for some $\lambda\in(0,1)$
\[
    \sup_{0<d(x,y)<\varepsilon}\bE \left(\log\frac{Z_{1,d}^{x,y}}{d(x,y)}\right)\leq \log\lambda<0.
\]

\begin{lemma}
	$\varepsilon$-LCA implies $\varepsilon$-local $\log$-LCA.
\end{lemma}

\begin{proof}
By Jensen's inequality,
$$\sup_{0<d(x,y)<\varepsilon}\bE \left(\log\frac{Z_{1,d}^{x,y}}{d(x,y)}\right)\leq \sup_{0<d(x,y)<\varepsilon}\log\bE \left(\frac{Z_{1,d}^{x,y}}{d(x,y)}\right). $$
Since logarithm is an increasing function, we get
$$\sup_{0<d(x,y)<\varepsilon}\bE \left(\log\frac{Z_{1,d}^{x,y}}{d(x,y)}\right)\leq\log\left( \sup_{0<d(x,y)<\varepsilon}\frac{\bE \left(Z_{1,d}^{x,y}\right)}{d(x,y)}\right), $$
which implies the lemma.
\end{proof}
The triple $(\cF,\bp,d)$ is \emph{locally contractive in the weak sense (LCWS)}, if for some $\lambda\in(0,1)$ it holds
\begin{align}\label{sten2}
    \sup_{x\in M}\left[\limsup_{y\to x}\frac{\bE (Z_{1,d}^{x,y})}{d(x,y)}\right]\leq \lambda.
\end{align}
It is clear that \eqref{sten1} implies \eqref{sten2}. In \cite[Remark 9]{sten} an example of a triple $(\cF,\bp,d)$ that is $\varepsilon$-local $\log$-CA but not LCWS is shown.

\begin{lemma}
	$\varepsilon$-LCA implies LCWS and ESCA.
\end{lemma}

\begin{proof}
The first implication is immediate. To check the second one, take $k=1$  and $V_{(x,x)}=\{(y,z)\colon d(y,z)<\varepsilon\}$ for every $x\in M$.
\end{proof}

\section{Conditions to guarantee CA}\label{sec:CondCA}

Let us begin this section by presenting a sufficient condition for the existence of a metric $D$  equivalent to $d$ for which the system is CA. We always consider an IFS $\cF$ of a metric space $(M,d)$. 

\begin{proposition}\label{pro:3}
Assume that  $(\cF,\bp,d)$ is SA${}_{\rm exp}$, that is, there exist constants $C>0$ and $\lambda\in(0,1)$ such that 
\begin{align}\label{contccc}
	\bE (Z_{n,d}^{x,y})\leq C\lambda^n
                \quad\text{for every $x,y\in M$ and $n\in \bN$}.	
\end{align}
For every $q\in(\lambda,1)$ 
\[	D(x,y)
	\eqdef 
	\sum_{n\geq 0}\frac{q^n}{\lambda^n}\bE (Z_{n,d}^{x,y}).
\]
defines a metric on $M$ which is equivalent to $d$ such that $(\cF,\bp,D)$ is CA.
\end{proposition} 

\begin{proof} 
	Pick $q\in(\lambda,1)$ and define $D\colon M\times M\to \mathbb{R}$ by \eqref{contccc}.
Note that it follows from our hypothesis that $D$ is well-defined. It is easy to check that $D$ is a metric. It remains to show that $D$ has the claimed properties.  By Lemma \ref{lemaABC1}, for every $x,y\in M$
\[\begin{split}
     \sum_{i=0}^{N-1}p_iD(f_i(x),f_i(y))
     &=\sum_{n\geq 0}\frac{q^n}{\lambda^n}\sum_{i=0}^{N-1}
     	p_i \bE \big(Z_{n,d}^{f_i(x),f_i(y)}\big)
     =\sum_{n\geq 0}\frac{q^n}{\lambda^n}\bE (Z_{n+1,d}^{x,y})\\
     &=\frac{\lambda}{q} \sum_{n\geq 1}
     	\frac{q^n}{\lambda^n}\bE (Z_{n,d}^{x,y})
     \leq \frac{\lambda}{q} \sum_{n\geq 0}
     	\frac{q^n}{\lambda^n}\bE (Z_{n,d}^{x,y})
     =\frac{\lambda}{q}D(x,y).
 \end{split}\]
Hence, $(\cF,\bp,D)$ is CA with contraction rate $\lambda/q\in(0,1)$.

It remains to see that $D$ and $d$ are  equivalent. First note that $d\le D$, which implies that the topology of $(M,d)$ is a subset of the topology of $(M,D)$. Now, let us prove that the topology of $(M,D)$ is a subset of the topology of $(M,d)$. Given $V$ an open set of $(M,D)$ and $x\in V$, there exists $r>0$ such that $B_D(x,r)\subset V$, where $B_D(x,r)$ is the open ball relative to the metric $D$ with center $x$ and radius $r$. Let $\varepsilon\eqdef (1-q)r/(2-q)>0$. Take $L\in\mathbb{N}$ such that 
\[
	C\sum_{n\geq L}q^n<\varepsilon.
\]	
It follows from Lemma \ref{lemaABC1}, that 
\[
	U
	\eqdef \bigcap_{n=0}^{L-1}U_n,\quad\text{ where }\quad
	U_n=\{y\in M: \bE (Z_{n,d}^{x,y})<\lambda^n\varepsilon\}
\]
is an open set of $(M,d)$. Furthermore, for every $y\in U$
\[
	D(x,y)
	\leq \sum_{n=0}^{L-1}\frac{q^n}{\lambda^n}\bE (Z_{n,d}^{x,y})
		+ \sum_{n=L}^\infty\frac{q^n}{\lambda^n}\bE (Z_{n,d}^{x,y})
	<  \sum_{n=0}^{L-1}\varepsilon q^n +\varepsilon
	< \varepsilon\frac{1}{1-q}+\varepsilon=r.
\]
Thus, $x\in U\subset B_D(x,r)$ and hence $x\in U\subset V$, which proves the desired. Therefore, the topologies of $(M,d)$ and  $(M,D)$ coincide.
\end{proof}

\begin{remark}
Proposition \ref{pro:3} can be applied to IFS $\cF$ on $\bS^1$ induced by the projective action of $GL(2,\bR)$ matrix cocycles and implies the  existence of a metric $D$ that makes $(\cF,\bp,D)$ CA. We refrain from providing the details. In this context, the existence of a unique stationary measure is well known (see, for example, \cite[Chapter II]{PJ85}) and no further immediate application of CA is given.

However, in this respect, it is reasonable to ask if the existence of a unique stationary measure implies the existence of some metric that preserves the topology and makes the system CA?
\end{remark}

\begin{remark}
If $(\cF,\bp,d)$ is not ESCA, then $D$ provided by Proposition \ref{pro:3}  is in general not strongly equivalent. Indeed, if ESCA fails, then there exists $x\in K$ such that for every $n\geq 1$ there exists  a sequence $\{x_k^n\}_{k\in\mathbb{N}}$ in $M$ such that $\lim_{k\to\infty}x_k^n= x$ in $(M,d)$ and 
\[
	\lim_{k\to\infty}\frac{\bE (Z_{n,d}^{x,x^n_k})}{d(x,x^n_k)}\geq1.
\]	
Therefore, we can find a sequence $\{y_n\}_{n\in\bN}$ such that
\[
	\lim_{n\to\infty} \frac{D(x,y_n)}{d(x,y_n)}=\infty,
\]	
\end{remark}

We now invoke the results obtained in Sections \ref{sec:22} and \ref{sec:23} to prove Theorem \ref{pro:1}.

\begin{proof}[Proof of Theorem \ref{pro:1}]
Assume that $(\cF,\bp,d)$ is NEA, LECA, and ESCA.

First, let us construct  for every $(x,y)\in K\times K$ an open neighborhood $V_{(x,y)}$ as follows. For $(x,y)\in K\times K$, $x\neq y$, take $\ell=\ell(x,y)\in\bN$ and $\lambda(x,y)\in(0,1)$ as in the definition of  (see \eqref{ip}) satisfying
\[
	\bE(Z_{\ell,d}^{x,y})\le \lambda(x,y) d(x,y)
\]
and let
\[
	V_{(x,y)}
	\eqdef \left\{ (z,w)\in K\times K\colon 
		z\neq w \mbox{ and } \frac{\bE (Z_{\ell,d}^{z,w})}{ d(z,w)}
							<\sqrt{\lambda(x,y)} \right\}.
\]
As the function 
\[
	f\colon K\times K\setminus\{(z,z)\colon z\in K\}\to \mathbb{R},\quad
	f(z,w)\eqdef\frac{\bE (Z_{\ell,d}^{z,w})}{ d(z,w)},
\]	
 is continuous and the set $\{(z,z)\colon z\in K\}$ is closed, $V_{(x,y)}$ is an open subset of $K\times K$ containing $(x,y)$.
For $x=y\in K$, take $\ell=\ell(x,x)\in\bN$ and an open neighborhood $V_{(x,x)}$ of $(x,x)$ as in the definition of ESCA (see \eqref{iip}) satisfying
\[
	\lambda(x,x)
	\eqdef \sup_{(z,w)\in V_{(x,y)}\colon y\neq z}
	\frac{\bE (Z_{\ell,d}^{z,w})}{ d(z,w)}<1
\]	
to get
    \[\bE (Z_{\ell,d}^{z,w})\leq \lambda(x,x) d(z,w),\]
    for all $(z,w)\in V_{(x,y)}$.
    
As $K\times K$ is compact, it has a finite sub-cover $\{V_{(x_1,y_1)},\ldots,V_{(x_m,y_m)}\}$. Hence, for every $i\in\{1,\ldots,m\}$ there are $\ell_i=\ell(x_i,y_i)\in\bN$ and $\lambda_i=\lambda(x_i,y_i)\in(0,1)$ such that for all $(z,w)\in V_{(x_i,y_i)}$ it holds
\[
	\bE (Z_{\ell_i,d}^{z,w})\leq \lambda_i d(z,w).
\]	
Take $k\eqdef\max_{1\leq i\leq m}\ell_i$ and $\lambda\eqdef\max_{1\leq i \leq m} \lambda_i$. Hence, together with Lemma \ref{lemaABC} (2), for every $(x,y)\in K\times K$ there exists $i\in\{1,\ldots,m\}$ such that
\[
	\bE (Z_{k,d}^{x,y})
	\le \bE (Z_{\ell_i,d}^{x,y})
	\leq \lambda_i d(x,y)
	\leq \lambda d(x,y).
\]	
But this implies that $(\cF,\bp,d)$ is $k$-ECA with contraction rate $\lambda\in(0,1)$. 
By Proposition \ref{prolem1:1}, $(\cF,\bp,D)$, where $D\eqdef d_{k,\lambda}$ is defined in \eqref{metricaECA}, is CA with contraction rate $\lambda^{1/k}$.
Hence, invoking Lemma \ref{lemrem:strequ}, $d$ and $D$ are strongly equivalent.

Together with Lemma \ref{lem:2}, this proves the theorem.
\end{proof}

The following provides a necessary condition for the existence of a metric $ D $  equivalent to $ d $ for which the system is CA. 

\begin{proposition}\label{pro:2}
Suppose that $(\cF,\bp,d)$ is NEA on some compact metric space $(K,d)$. 
If there exists a metric  $D$ on $K$  equivalent to $d$ such that $(\cF,\bp,D)$ is CA, then $(\cF,\bp,d)$ is SA.
\end{proposition} 

\begin{proof}
Assuming that $(\cF,\bp,D)$ is CA, by Lemma \ref{lemalambda} for every $n\in\bN$ it holds
\[
    \bE (Z_{n,D}^{x,y})\leq \lambda^n D(x,y)
    \quad\text{for every $x,y\in K$}.
\]

 Fix $x,y\in K$. By the above, it holds
\[
    \lim_{n\to \infty}\bE (Z_{n,D}^{x,y})=0,
\]
 that is, $Z_{n,D}^{x,y}$ converges to $0$ in $L^1$. By Chebyshev’s inequality, for every $\varepsilon>0$ it holds
\[
	\mu(Z_{n,D}^{x,y}\geq \varepsilon)
	\leq \varepsilon^{-1}\bE (Z_{n,D}^{x,y}).
\]	
and hence it follows that $Z_{n,D}^{x,y}$ converges to $0$ in probability. By \cite[Theorem 2.3.2]{Durrett}, there exists a sub-sequence $(n_k)_k$ such that $Z_{n_k,D}^{x,y}$ converges almost
surely to 0 as $k\to \infty$.
 
The fact that $D$ and $d$ are  equivalent implies $Z_{n_k,d}^{x,y}$ converges almost
surely to 0 as $k\to \infty$. By dominated convergence theorem, we conclude that 
\[
	\lim_{k\to\infty}\bE (Z_{n_k,d}^{x,y})= 0.
\]
As we assume that $(\cF,\bp,d)$ is NEA on a compact space and $x,y$ were arbitrary, by Lemma \ref{lemaABC} (2)--(3) it follows
\[
	\lim_{n\to\infty}\bE (Z_{n,d}^{(\cdot),(\cdot)})=0
\]	
uniformly. 
This implies SA.
\end{proof}


\section{CA for IFS's on $\bS^1$}\label{sec:5}

In this section, we will study the particular case of an IFS $\mathcal{F}$ of homeomorphisms on $K=\bS^1$ (equipped with the usual metric $d(x,y)\eqdef\min\{\lvert x-y\rvert,1-\lvert x-y\rvert\}$. In particular, we prove Theorem \ref{teo3.1}. 

We first recall the following results which are an immediate consequence of \cite[Theorem A and Proposition 4.2]{Malicet}, respectively.

\begin{proposition}\label{P1}
Let $\cF$ be an IFS of homeomorphisms of $\bS^1$ and assume that there does not exist a probability measure which is invariant by every element of $\mathcal{F}$. Then for every non-degenerate probability vector $\bp$ there is a constant $\lambda\in(0,1)$ such that for every $x\in\bS^1$ and almost every $\xi\in\Sigma_N^{+}$ there exists an open neighborhood $I_x(\xi)\subset\bS^1$ of $x$ such that for all $ n\in\mathbb{N}$ it holds
\[	 Z_{n,d}^{w,z}(\xi)
	 =d(f_{\xi}^n(w),f_{\xi}^n(z))\leq \lambda^n
	     \quad\text{for every $w,z\in I_x(\xi)$}.
\]\end{proposition}

For the statement of the next proposition, consider the \emph{shift} $\sigma\colon\Sigma_{N}^{+}\to\Sigma_{N}^{+}$ defined by 
\[    (\sigma(\xi))_{j}=\xi_{j+1},\quad j\geq 1.
\]Recall that it is continuous.

\begin{proposition}\label{P2}
Under the hypotheses of Proposition \ref{P1}, consider the map 
\[
	G\colon\Sigma_N^{+}\times \bS^1\times \bS^1
		\to \Sigma_N^{+}\times \bS^1\times \bS^1
	,\quad	 
	G(\xi,x,y)
	\eqdef (\sigma(\xi),f_{\xi_1}(x),f_{\xi_1}(y)).
\]	 
Let $\mathcal{E}=\bigcup_{\xi\in\Sigma_N^{+}}\{\xi\}\times U(\xi)\subset \Sigma_N^{+}\times \bS^1\times \bS^1$ such that $G^{-1}(\mathcal{E})\subset \mathcal{E}$ and $U(\xi)$ is open in $\bS^1\times \bS^1$ for every $\xi\in\Sigma_N^{+}$. Let $\bp$ be a non-degenerate probability vector and $\mu$ its its associate Bernoulli measure $\mu$ on $\Sigma^+_N$ and assume that
\[
	\left(\mu\otimes \nu\right) (\mathcal{E})>0
\]	
for every stationary%
\footnote{Recall that here $\nu$ is \emph{stationary} if and only if $\mu\otimes\nu$ is invariant by the skew product $G$.} probability measure $\nu$ on $\bS^1\times \bS^1$. Then actually,
\[
	\left(\mu\otimes \nu\right) (\mathcal{E})=1
\]	
for every probability measure $\nu$ on $\bS^1\times \bS^1$ (not necessarily stationary).
\end{proposition}

Now, let us prove the following lemma.

\begin{lemma}\label{lemcontxy}
	Assume that $(\cF,\bp,d)$ and $\lambda\in(0,1)$ are as in Proposition \ref{P1}. If $(\cF,d)$ is proximal, then
\[
	\mu(\Omega^{x,y})=1
        \quad\text{for every $x,y\in \bS^1$},
\]
where
\begin{equation}\label{Omegaxy}
     \Omega^{x,y}
     \eqdef \{\xi\in\Sigma_{N}^{+}\colon
     	\mbox{ there exists $C>0$ such that }Z^{x,y}_{n,d}(\xi)\leq C\lambda^n \text{ for all } n\in\bN\}.
 \end{equation}
\end{lemma}

\begin{proof}
 For every $z\in\bS^1$ and $k\in\mathbb{N}$ let 
\[
	\Gamma_k(z)
	\eqdef \Big\{\eta\in\Sigma_N^{+}\colon 
	d(f_{\eta}^n(z_1),f_{\eta}^n(z_1))\leq \lambda^n
	\mbox{ for all }n\in\mathbb{N}, z_1, z_2\in \Big(z-\frac{1}{k},z+\frac{1}{k}\Big)	
			\Big\}.
\]	
Clearly, $\Gamma_k(z)\subset \Gamma_{k+1}(z)$.
By Proposition \ref{P1}, for every $z\in\bS^1$ 
\[
	\mu\Big(\bigcup_{k\in\mathbb{N}}\Gamma_k(z)\Big)
	=1.
\]	 
Hence, there is $k_0=k_0(z)\in\mathbb{N}$ such that
\begin{equation}\label{akzeropositivo}
    \mu\left(\Gamma_{k_0}(z)\right)>0. 
\end{equation}

Let $\mathcal{E}$ be the set of points $(\xi,x,y)\in\Sigma_N^{+}\times\bS^1\times\bS^1$ such that there exist $z\in\bS^1$, $k_0\in\mathbb{N}$, and $k_1\in\mathbb{N}$  satisfying
\[
 	f_{\xi}^{k_1 }(x),f_{\xi}^{k_1 }(y)
	\in\left(z-\frac{1}{k_0},z+\frac{1}{k_0}\right),\quad
 \xi\in\sigma^{-k_1}\left(\Gamma_{k_0}(z) \right)\quad \mbox{and}
 \quad\mu\left(\Gamma_{k_0}(z)\right)>0.
\] 
By the following claim, $\mathcal{E}$ is nonempty. More precisely, for every $(x,y)\in\bS^1\times\bS^1$ the set $\mathcal{E}\cap \left(\Sigma_N^{+}\times\{(x,y)\}\right)$ is nonempty.

\begin{claim}
	For every $x,y\in\bS^1$ it holds $\left(\mu\otimes \delta_{(x,y)}\right)(\mathcal{E})>0$, where $\delta_{(x,y)}$ is the Dirac measure at $(x,y)$.
\end{claim}

\begin{proof}
Fix $x,y\in\bS^1$. By proximality, there exist $\xi\in\Sigma_N^{+}$ and an increasing sequence $(n_k)_{k\in\mathbb{N}}$ such that  we have
\[
    \lim_{k\to\infty}Z_{n_k,d}^{x,y}(\xi)=0.
\]
By compactness of $\bS^1$,  there are $z\in\bS^1$ and a subsequence $(n_{k_j})_{j\geq1}$ such that
\[
	f_{\xi}^{n_{k_j}}(x)\to z \quad \text{and} \quad f_{\xi}^{n_{k_j}}(y)\to z, 
\]	
as $j\to\infty$. Hence, taking $k_0=k_0(z)$ as in \eqref{akzeropositivo}, there exists $k_1 \in\mathbb{N}$ large enough such that $f_{\xi}^{k_1 }(x)$ and $f_{\xi}^{k_1 }(y)$ are both in $\left(z-\frac{1}{k_0},z+\frac{1}{k_0}\right)$. 

As $\mu$ is $\sigma$-invariant, it holds $\mu( \sigma^{-k_1 }( \Gamma_{k_0}(z))) =\mu( \Gamma_{k_0}(z))>0$. Furthermore,
\[
	\sigma^{-k_1 }\left( \Gamma_{k_0}(z)\right)
	=\underbrace{\Sigma^+_N\times\cdots\times\Sigma^+_N\,}_\text{$k_1$-times} \times\Gamma_{k_0}(z),
\]	
which implies that $\mu\left([\xi_1,\ldots,\xi_{k_1 }]\cap\sigma^{-k_1 }\left( \Gamma_{k_0}(z)\right) \right)>0$. Since 
\[
	\left([\xi_1,\ldots,\xi_{k_1 }]\cap
		\sigma^{-k_1 }\left( \Gamma_{k_0}(z)\right)\right)\times\{(x,y)\}
	\subset\mathcal{E},
\] 
this implies the claim.
\end{proof}

Integrating over $(x,y)\in\bS^1\times\bS^1$ with respect to any stationary probability measure $\nu$ on $\bS^1\times \bS^1$, it follows that
\[
	\left(\mu\otimes \nu\right) (\mathcal{E})>0.
\]	

\begin{claim}
The set $\mathcal{E}$ is $G$-invariant, that is,
$G^{-1}(\mathcal{E})\subset \mathcal{E}$.
\end{claim}  
\begin{proof}
If $(\xi,x,y)\in G^{-1}(\mathcal{E})$ then $(\sigma(\xi),f_{\xi_1}(x),f_{\xi_1}(y))\in\mathcal{E}$. Hence there are $k_1 ,k_0\in\bN$ and $z\in\bS^1$ satisfying
\[\begin{split}
 	&f_{\sigma(\xi)}^{k_1 }(f_{\xi_1}(x)),f_{\sigma(\xi)}^{k_1 }(f_{\xi_1}(y))
	\in \Big(z-\frac{1}{k_0},z+\frac{1}{k_0}\Big),\\
	&\sigma(\xi)\in\sigma^{-k_1}(\Gamma_{k_0}(z)),\quad
	\mu\left(\Gamma_{k_0}(z)\right)>0.
\end{split}\]	 
As $f_{\sigma(\xi)}^{k_1 }\circ f_{\xi_1}=f^{k_1 +1}_{\xi}$, this implies
\[\begin{split}
 	&f_{\xi}^{k_1 +1}(x),f_{\xi}^{k_1 +1}(y)
		\in \Big(z-\frac{1}{k_0},z+\frac{1}{k_0}\Big),\\
	&\xi\in\sigma^{-(k_1+1)}(\Gamma_{k_0}(z)),\quad
	\mu\left(\Gamma_{k_0}(z)\right)>0.
\end{split}\]	 
But this implies  $(\xi,x,y)\in\mathcal{E}$.
\end{proof}

\begin{claim}
For every $\xi\in\Sigma_N^{+}$ the set
$U({\xi})\eqdef \{(x,y)\in\bS^1\times\bS^1\colon
	(\xi,x,y)\in\mathcal{E}\}$		 
is open in $\bS^1\times\bS^1$.
\end{claim}  

\begin{proof}
Fix $\xi\in\Sigma_N^{+}$. Given $x,y\in\bS^1$ such that $(\xi,x,y)\in\mathcal{E}$, there exist $z\in\bS^1$ and $k_0 ,k_1\in\mathbb{N}$ satisfying
\[\begin{split}
    &f_{\xi}^{k_1 }(x),f_{\xi}^{k_1 }(y)
    	\in \Big(z-\frac{1}{k_0},z+\frac{1}{k_0}\Big),\\  
    &\xi\in\sigma^{-k_1}\Big( \Gamma_{k_0}(z)\Big),
    \quad 
    \mu\big( \Gamma_{k_0}(z)\big)>0.
\end{split}\]
The continuity of $f_{\xi}^{k_1 }$ implies that the set
\[
	V_{(x,y)}
	\eqdef \Big(f_{\xi}^{k_1}\Big)^{-1}\Big(z-\frac{1}{k_0},z+\frac{1}{k_0}\Big)
		\times \Big(f_{\xi}^{k_1}\Big)^{-1}
			\Big(z-\frac{1}{k_0},z+\frac{1}{k_0}\Big)
\]			
is an open neighborhood of $(x,y)$ in $\bS^1\times\bS^1$. For every $(w_1,w_2)\in V_{(x,y)}$, it holds $f_{\xi}^{k_1}(w_1)$, $f_{\xi}^{k_1}(w_2)\in \left(z-1/k_0,z+1/k_0 \right)$. Thus, $V_{(x,y)}\subset U(\xi)$. 
This proves the claim.
\end{proof}

By Proposition \ref{P2}, for every $x,y\in\bS^1$ we have 
$$\left(\mu\otimes \delta_{(x,y)}\right)(\mathcal{E})=1.$$
Therefore, by definition of the sets $\Gamma_{k_0}(z)$ for every $x,y\in\bS^1$ and almost every $\xi\in\Sigma_N^{+}$ there exists $k_1\in\mathbb{N}$ such that
\[
   Z_{n+k_1 ,d}^{x,y}(\xi)
   =d(f_{\xi}^{k_1 +n}(x),f_{\xi}^{k_1 +n}(y))
   \leq \lambda^n.
\]
This proves the lemma.
\end{proof}

The following is an immediate consequence of Lemma \ref{lemcontxy} by the dominated convergence theorem.
 
\begin{corollary}\label{correm:5:7}
	Assume that $(\cF,\bp,d)$ and $\lambda\in(0,1)$ are as in Proposition \ref{P1}. Then $(\cF,\bp,d)$ is SA (and hence LECA).
\end{corollary}

In the rest of this section, $\mathcal{F}$ is a finite family of $C^1$-diffeomorphisms. Hence, there exist $L>1$ such that for all $i\in\{0,\ldots, N-1\}$
\begin{align}\label{abeexisten}
    	L^{-1} d(x,y)
	\leq d(f_i(x),f_i(y))
	\leq L d(x,y).
\end{align}

\begin{lemma}\label{nderivative}
Assume that $(\cF,\bp,d)$ and $\lambda\in(0,1)$ are as in Proposition \ref{P1}. 
Assume also that each map in $\mathcal{F}$ is a $C^1$-diffeomorphism.
Then, for every $t\in(\lambda,1)$ and $x\in\bS^1$ and almost every $\xi$ there exists an open neighborhood $J_x(\xi)\subset\bS^1$ of $x$ such that there exists $C>0$ satisfying for all $ n\in\bN$ we have
\[
	\max_{z\in J_x(\xi)}\lvert(f_{\xi}^n)'(z)\rvert
	\leq Ct^n.
\]	
\end{lemma}
\begin{proof}
Fix $t\in(\lambda,1)$ and $x\in\bS^1$. By Proposition \ref{P1}, for almost every $\xi\in\Sigma_N^{+}$ there exits an open arc $I_x(\xi)\subset\bS^1$ of $x$ such that for every $n\in\mathbb{N}$ and $y,z \in I_x(\xi)$
\begin{align}\label{cedilcedil1}
 	Z_{n,d}^{y,z}(\xi)\leq \lambda^n.
\end{align}
Denote by $\omega_{\xi}(\cdot)$ the modulus of continuity of $\log |f'_{\xi_1}|$. Since $\omega_{\xi}(\varepsilon)$ tends to $0$ as $\varepsilon\to 0$ and is uniformly bounded, by dominated convergence it follows
\[
	\lim_{\varepsilon\to 0}\int \omega_{\xi}(\varepsilon) d\mu(\xi)=0.
\]	
Fix $\varepsilon>0$ such that
\[
	\int \omega_{\xi}(\varepsilon) d\mu(\xi)
	\leq \log \frac{\lambda+t}{2} -\log \lambda. 
\]	
By Birkhoff ergodic theorem, almost every $\xi\in\Sigma_N^{+}$
\begin{align}\label{cedilcedil2}
    \lim_{n\to\infty}\frac{1}{n}\sum_{j=0}^{n-1}\omega_{\sigma^j(\xi)}(\varepsilon)
    = \int \omega_{\eta}(\varepsilon) d\mu(\eta).
\end{align}
Now, fix $\xi\in\Sigma_N^{+}$ such that \eqref{cedilcedil1} and \eqref{cedilcedil2} hold. Take $k\geq 1$ so that $\lambda^k<\varepsilon$. Define
\[ 
	I_x(\xi,\varepsilon)
	\eqdef  I_x(\xi)\cap\bigcap_{j=0}^{k} 
	(f^j_{\xi})^{-1}\left(f^j_{\xi}(x)-\frac{\varepsilon}{2},f^j_{\xi}(x)+\frac{\varepsilon}{2}\right),
\]	
and note that $ I_x(\xi,\varepsilon)$ is an open arc containing $x$. For every $y, z\in I_x(\xi,\varepsilon)$ it holds
\[
    \log\frac{\lvert(f_{\xi}^n)'(y) \rvert}{\lvert (f_{\xi}^n)'(z) \rvert} 
    =\sum_{j=0}^{n-1} \log\,\lvert f'_{\sigma^j(\xi)}(f^j_{\xi}(y)) \rvert 
    	-\log\,\lvert f'_{\sigma^j(\xi)}(f^j_{\xi}(z)) \rvert 
    \leq \sum_{j=0}^{n-1}\omega_{\sigma^j(\xi)}(\varepsilon).
\]
Let $x_1$ and $x_2$ be the extreme points of $I_x(\xi,\varepsilon)$. Note that for every $z\in I_x(\xi,\varepsilon)$, it holds
\[
	\frac{\log\,\lvert(f_{\xi}^n)'(z)  \rvert}n
	\leq\frac{1}{n}\log\left(\frac{Z_{n,d}^{x_1,x_2}}{d(x_1,x_2)}\right)
		+\frac{1}{n}\sum_{j=0}^{n-1}\omega_{\sigma^j(\xi)}(\varepsilon).
\]		
Hence, using \eqref{cedilcedil1} for all $n\geq 1$
\[
	\frac{1}{n}\log\left(\max_{z\in I_x(\xi,\varepsilon) }\lvert (f_{\xi}^n)'(z)  \rvert \right)	\leq\log \lambda-\frac{1}{n}\log d(x_1,x_2)+\frac{1}{n}\sum_{j=0}^{n-1}\omega_{\sigma^j(\xi)}(\varepsilon),
\]	
so that
\[
	\limsup_{n\to\infty}\frac{1}{n}
	\log\left(\max_{z\in I_x(\xi,\varepsilon) }
	\lvert (f_{\xi}^n)'(z)  \rvert \right)\leq\log \frac{\lambda+t}{2}.
\]	
Then, there exists $C>0$ such that for all $n\geq 1$
$$\max_{z\in I_x(\xi,\varepsilon)}|(f_{\xi}^n)'(z)|\leq Ct^n.$$
This proves the lemma.
\end{proof}

The following result together with Proposition \ref{prolem1:1} immediately implies Theorem \ref{teo3.1}.

\begin{proposition}\label{prolemmaexisteee}
Assume the hypotheses of Theorem \ref{teo3.1}. Then there exist  $\alpha\in(0,1)$, $\lambda\in(0,1)$, and $k\in\bN$ such that $(\cF,\bp,d^\alpha)$ is $k$-ECA with contraction rate $\lambda$.
\end{proposition}

\begin{proof}Fix points $x,y\in\bS^1$.
Let 
\[
	A_n
	\eqdef \sup_{x\neq y}\int_{\Sigma_N^{+}}
		\log\left(\frac{Z_{n,d}^{x,y}(\xi)}{d(x,y)} \right)d\mu(\xi) .
\]	
Note that $(A_n)_{n\geq 0}$ is a subadditive sequence. Hence, by Fekete's Lemma, the limit $A\eqdef \lim _{n\to \infty } A_{n}/n=\inf_{n\ge1} A_{n}/n\in[-\infty,\infty)$ exists.

All hypotheses of Proposition \ref{P1} are satisfied and we can consider $\lambda\in(0,1)$ as provided by this proposition.

\begin{claim}\label{claclaclacla}
	$A\leq \log \lambda$.
\end{claim}

\begin{proof}
Arguing by contradiction, suppose that $\log \lambda<A$. 
Then, for all $n\in\bN$ 
\[
	\log\lambda
	< A
	\le \sup_{x\neq y}\int_{\Sigma_N^{+}}F_n(x,y,\xi)
		\,d\mu(\xi),
	 \quad\text{ where }\quad
	F_n(x,y,\xi)
	\eqdef \frac{1}{n}\log\left(\frac{Z_{n,d}^{x,y}(\xi)}{d(x,y)} \right). 
\]		
Thus, for all $n\in\bN$ there exist $x_n,y_n$ in $\bS^1$, $x_n\neq y_n$, such that 
\begin{equation}\label{seqdepontos}
    A
    \le \int_{\Sigma_N^{+}}F_n(x_n,y_n,\xi)\,
    	d\mu(\xi).
\end{equation}
By compactness, there exist a subsequence $(n_k)_{k\geq 1}$ and points $x,y\in\bS^1$ such that 
\begin{equation}\label{eq:convergence}
	\lim_{k\to\infty}x_{n_k}=x,
	\quad\quad
	\lim_{k\to\infty}y_{n_k}=y.
\end{equation}
In the following two cases we consider $\xi$ in a appropriate set of measure $1$ to obtain that the limit superior of $F_{n_k}(x_{n_k},y_{n_k},\xi)$ as $k\to\infty$  is less than or equal to $\log \lambda$. We then will apply Fatou's Lemma to contradict \eqref{seqdepontos}. 
Note that all hypotheses of Lemmas \ref{lemcontxy} and \ref{nderivative} are satisfied. Given $x,y$ as above, let $\Omega^{x,y}$ be as in \eqref{Omegaxy}.

\medskip\noindent
\textbf{Case $x\neq y$.} Fix any $t\in(\lambda,1)$. Denote by $\Gamma_t$ the set of sequences $\xi\in\Sigma_N^{+}$ such that there exist 
$C>0$ and open arcs $J_x(\xi)$ and $J_y(\xi)$ containing $x$ and $y$, respectively, and satisfying
\[
	\max_{z\in J_x(\xi)\cup J_y(\xi)}\lvert(f_{\xi}^n)'(z)\rvert
	\leq Ct^n, \quad Z^{x,y}_{n,d}(\xi)\leq C\lambda^n \text{ for all } n\in\bN
\]
By Lemma \ref{lemcontxy} and Lemma \ref{nderivative}, the set $\Gamma_t$ has measure 1. Given $\xi\in\Gamma_t$, by the triangle inequality and the mean value inequality it follows
\begin{align*}
    Z_{n_k,d}^{x_{n_k},y_{n_k}}(\xi)
    &\leq Z_{n_k,d}^{x_{n_k},x}(\xi) +Z_{n_k,d}^{x,y}(\xi)+Z_{n_k,d}^{y,y_{n_k}}(\xi)\\
    &= d(f_\xi^{n_k}(x_{n_k}),f_\xi^{n_k}(x))
    	+d(f_\xi^{n_k}(x),f_\xi^{n_k}(y))
	+d(f_\xi^{n_k}(y),f_\xi^{n_k}(y_{n_k}))\\
    &\leq \lvert(f_{\xi}^{n_k})'(\hat x_k)\rvert d(x_{n_k},x)+Z_{n,d}^{x,y}(\xi)
    	+\lvert(f_{\xi}^{n_k})'(\hat y_k)\rvert d(y_{n_k},y),
\end{align*}
for some points $\hat x_k$ and $\hat y_k$ between $x_{n_k}$ and $x$ and between $y_{n_k}$ and $y$, respectively. By \eqref{eq:convergence}, $\hat x_k\to x$ and $\hat y_k\to y$ as $k\to\infty$. Then, for $k$ large enough $\hat x_k\in J_x(\xi)$ and $\hat y_k\in J_y(\xi)$ and it follows
\[
    Z_{n_k,d}^{x_{n_k},y_{n_k}}(\xi)  
    \leq Ct^{n_k} d(x_{n_k},x)
    	+C\lambda^{n_k}
		+Ct^{n_k} d(y_{n_k},y).
\]
This implies
\[
	\limsup_{k\to\infty}F_{n_k}(x_{n_k},y_{n_k},\xi)
	\leq \log t.
\]	
Given $L>1$ satisfying \eqref{abeexisten}, as
\[
	-\log L
	\leq F_{n_k}(x_{n_k},y_{n_k},\xi)
	\leq \log t
	<0,
\]	
Fatou's Lemma implies 
\[
    \limsup_{k\to \infty}\int  F_{n_k}(x_{n_k},y_{n_k},\xi)\,
	d\mu(\xi)
  	\leq \int \limsup_{k\to \infty}F_{n_k}(x_{n_k},y_{n_k},\xi)\,
	d\mu(\xi)
    \leq \log t.
\]
As $t\in(\lambda,1)$ was arbitrary, it follows
\[
	 \limsup_{k\to \infty}\int  F_{n_k}(x_{n_k},y_{n_k},\xi)d\mu(\xi)
	\leq \log \lambda,
\]	
which contradicts \eqref{seqdepontos}.

\medskip\noindent
\textbf{Case $x=y$.} By Lemma \ref{nderivative}, for every $t\in(\lambda,1)$ and  $x\in\bS^1$ and for almost every $\xi$ there exist an open neighborhood $J_x(\xi)\subset\bS^1$ of $x$ and a constant $C>0$ such that for all $ n\in\mathbb{N}$ it holds
\begin{align}\label{derderder}
  \max_{z\in J_x(\xi)}|(f_{\xi}^n)'(z)|\leq C t^n.  
\end{align}
Fix $\xi\in\Sigma_N^{+}$ such that \eqref{derderder} holds. Take $k_1\in\bN$ large enough so that $x_{n_k}$ and $y_{n_k}$ are both $J_x(\xi)$, for all $k\geq k_1$. By the mean value inequality,
\[
	F_{n_k}(x_{n_k},y_{n_k},\xi)
	\leq \frac{1}{n_k}\log C+ \log t,
\]	
for all $k\geq k_1 $. Thus,
\[
	\limsup_{k\to\infty}F_{n_k}(x_{n_k},y_{n_k},\xi)
	\leq \log t.
\]	
By Fatou's lemma and using again that $t\in(\lambda,1)$ was arbitrary, it follows
\begin{equation*}
    \limsup_{k\to\infty} \int F_{n_k}(x_{n_k},y_{n_k},\xi)
    \,d\mu(\xi)
    \leq \int \limsup_{k\to\infty}F_{n_k}(x_{n_k},y_{n_k},\xi)
    \, d\mu(\xi)
    \leq \log \lambda,
\end{equation*}
which contradicts \eqref{seqdepontos}.

This proves the claim.
\end{proof}

By Claim \ref{claclaclacla}, for every $n\in\bN$ sufficiently large,
\[
	\frac1nA_n
	< \frac{1}{2}\log \lambda<0.
\]	
Consider again $L>1$ as in \eqref{abeexisten}. 
Hence, for every $n\in\bN$, $\xi$, and $x\ne y$ 
\[
	-n\log L
	\le \left\lvert\log\frac{d(f_{\xi}^n(x),f_{\xi}^n(y))}{d(x,y)}\right\rvert
	\le n\log L
\]
Using the above and that $e^x\leq 1+x+x^2e^{|x|}/2$, for every $\alpha\in(0,1)$ it follows that
\begin{align*}
    \int_{\Sigma_N^{+}}&\displaystyle
    	\frac{d^{\alpha}(f_{\xi}^n(x),f_{\xi}^n(y))}{d^{\alpha}(x,y)}\,d\mu(\xi)
    =\int_{\Sigma_N^{+}}e^{\alpha\log(d(f_{\xi}^n(x),f_{\xi}^n(y))/d(x,y))}	
    	\,d\mu(\xi)\\
    &\leq \int_{\Sigma_N^{+}}
    	\Big(1
		+\alpha \log\frac{d(f_{\xi}^n(x),f_{\xi}^n(y))}{d(x,y)}
		+\\
	&\phantom{\le\int_{\Sigma_N^+}\Big(1}		
		+\frac{\alpha^2}{2}
		\log^2\left(\frac{d(f_{\xi}^n(x),f_{\xi}^n(y))}{d(x,y)}\right)
		e^{\lvert \log(d(f_{\xi}^n(x),f_{\xi}^n(y))/d(x,y))\rvert }\Big)
		\, d\mu(\xi)\\
	&\le 1 +\alpha\int_{\Sigma_N^{+}}	\log\left(\frac{Z_{n,d}^{x,y}(\xi)}{d(x,y)} \right)d\mu(\xi)
		+\frac{\alpha^2}{2}(n\log L)^2L^n
	\\
    &\le 1+\alpha A_n+\frac{\alpha^2}{2} (n\log L)^2 L^n
    < 1+\frac{\alpha}{2} n\log \lambda+\frac{\alpha^2}{2} (n\log L)^2 L^n.
\end{align*}
Now taking $\alpha\in(0,1)$ sufficiently small, the right hand side provides a contraction rate in $(0,1)$.
This implies the assertion.
\end{proof}

\begin{proof}[Proof of Theorem \ref{teo3.1}]
 Let $\alpha\in(0,1)$, $\lambda\in(0,1)$, and $k\in\bN$ be as in Proposition \ref{prolemmaexisteee} so that $(\cF,\bp,d^\alpha)$ is $k$-ECA with rate $\lambda$. 
By Proposition \ref{prolem1:1}, $(\cF,\bp,D)$  is CA with contraction rate $\lambda^{1/k}$, where 
\[
	D(x,y)
	\eqdef d^\alpha(x,y)
		+\frac{1}{\lambda^{1/k}}\bE(Z_{1,d^\alpha}^{x,y})
		+\cdots+\frac{1}{\lambda^{(k-1)/k}}\bE (Z_{k-1,d^{\alpha}}^{x,y}) 
\] 
is as in \eqref{metricaECA} for $d^\alpha$ instead of $d$. By Lemma \ref{lemrem:strequ}, it holds $d\le D\le Cd^\alpha$ for some $C>0$ and hence $D$ is strongly equivalent to $d^\alpha$. 
\end{proof}

\section{Examples}\label{sec:6}

The following examples illustrate that the hypotheses in Theorems \ref{pro:1} and \ref{teo3.1} are sharp. We consider $\bS^1$ equipped with the usual metric $d(x,y)\eqdef\min\{\lvert x-y\rvert,1-\lvert x-y\rvert\}$.

\subsection{LECA, but not ESCA}\label{sec:exe1}

Let $\cF=\{f_0,f_1\}$ be the family of two diffeomorphisms of $\bS^1$ such that $f_0$ has two fixed points, one attracting and one repelling, and $f_1$ is an irrational rotation. Note that $(\cF,d)$ is proximal.
Let $ \bp=(p,1-p)$ be a non-degenerate probability vector. Note that $(\cF,\bp,d)$ fails to be NEA.
The triple  $(\cF,\bp,d)$ satisfies the hypotheses of Theorem \ref{teo3.1} (and hence of Proposition \ref{P1}). Hence, by Corollary \ref{correm:5:7}, $(\cF,\bp,d)$ is SA and LECA.
It is easy to check that for $(\cF,\bp,d)$ condition \eqref{eq2} for NEA and condition \eqref{eq01} for $k$-ECA ($k\in\bN$ arbitrary) are violated at the repelling fixed point of $f_0$. 

Note that $(\cF,\bp,d)$ verifies the hypotheses of Theorem \ref{teo3.1}. Hence, there  exist $\alpha\in(0,1]$, $\lambda\in(0,1)$ and $n\in\mathbb{N}$ such  that for $D\eqdef(d^{\alpha})_{n,\lambda}$ the triple $(\cF,\bp,D)$ is CA. In particular, $(\cF,\bp,D)$ satisfies  NEA, SA, LECA, and ESCA. 
However, by the latter together with Lemma \ref{Lem:1}, for all $\beta\in(0,1]$ the metric $D^{\beta}$ cannot be strongly equivalent to $d$.

Let $\mu$ be a stationary probability for IFS with probabilities $(\cF^{-1},\bp)$. 
By stationarity, it holds
\[
	\mu=p\, (f_0^{-1})_\ast\mu+(1-p)\,(f_1^{-1})_\ast\mu.
\]
By \cite[Lemma 2.6]{KatSten}, $\mu$ is non-atomic and has full support. Consider the metric $\rho$ on $\bS^1$, given by $\rho(x,y)\eqdef\min\{\mu([x,y]),\mu([y,x])\}$. By \cite[Proposition 1.2]{KatSten},  $(\cF,\bp,\rho)$ is NEA.

\begin{lemma}
	 $(\cF,\bp,\rho)$ is SA and LECA.
\end{lemma}

\begin{proof}
Since $(\cF,\bp,d)$ is SA and $(\cF,\bp,\rho)$ is NEA, Lemma \ref{lemrem:SA} implies that $(\cF,\bp,\rho)$ is SA. By Lemma \ref{LECAimpl}, $(\cF,\bp,\rho)$ is LECA.
\end{proof}

The following result checking that the IFS is ``$\rho$-isometric in average'' if an only if it is ``$\rho$-isometric'' is straightforward.

\begin{claim}\label{cex1}
	Assume $\rho(x,y)=\mu([x,y])$. Then, it holds $\mathbb{E}(Z_{n,\rho}^{x,y})=\rho(x,y)$ if only if for all $\xi_1,\ldots,\xi_n$ we have
\[
     \rho(f_{\xi_1\ldots\xi_n}(x),f_{\xi_1\ldots\xi_n}(y))
     =\mu(f_{\xi_1\ldots\xi_n}([x,y])).
\]
\end{claim}

\begin{lemma}
The triple $(\cF,\bp,\rho)$  is not ESCA.
\end{lemma}

\begin{proof}
Given $x\in\bS^1$ and $\varepsilon\in(0,1)$, denote by $\arc(x,\varepsilon)$ the open arc centered at $x$ and with $\mu$-measure equal to $\varepsilon$ (recall that $\mu$ is nonatomic and has full support, hence $\arc(x,\varepsilon)$ is a nontrivial interval).  
Given $x$ and $\ell\in\mathbb{N}$, consider the set
\[
	V_{\ell}(x)
	\eqdef \arc(x,4^{-1})\cap 
		\bigcap_{\xi_1,\ldots,\xi_{\ell}}\left(f_{\xi_1\ldots\xi_\ell}\right)^{-1}
		\arc(f_{\xi_1\ldots\xi_\ell}(x),4^{-1})
		\subset\arc(x,4^{-1}),
\]	
which is also a nontrivial open interval. For every $y,z\in V_{\ell}(x)$ such that $[y,z]\subset V_{\ell}(x)$
for every $\xi_1,\ldots,\xi_{\ell}\in\{0,1\}$, it hence holds
\[
	f_{\xi_1\ldots\xi_\ell}([y,z])
	\subset\arc(f_{\xi_1\ldots\xi_\ell}(x),4^{-1})
\]	
and therefore
\[
	4^{-1}
	\ge \mu(f_{\xi_1\ldots\xi_\ell}([y,z]))\
	=  \rho(f_{\xi_1\ldots\xi_\ell}(y),f_{\xi_1\ldots\xi_\ell}(z)).
\]
Hence, from Claim \ref{cex1}, it follows
\begin{equation}\label{Ciins}
    \rho (y,z) 
    = \mu([y,z])
    = \mathbb {E}(Z_{\ell,\rho}^{y,z}).
\end{equation}

Any other neighborhood $V$ of $x$, contains an open arc $W$ containing $x$, so that for all $\ell\geq 1$, $W \cap V_{\ell} (x)\subset V$ is and open arc containing $x$. Now it is enough to consider $y,z\in W\cap V_{\ell} (x)$ to get \eqref{Ciins}. This completes the proof that $(\cF,\bp,\rho)$ is not ESCA.
\end{proof} 

\subsection{LECA, but not NEA}\label{sec:exe2}

Let  $\bp=(p,1-p)$ be a non-degenerate probability vector and $\mu$ its associated Bernoulli measure. Without loss of generality, we can assume $p=\max\{p,1-p\}$. Let $f_0,f_1\colon \bS^1\rightarrow\bS^1$ be orientation preserving homeomorphisms, such that there exist two open arcs $I,J\subset \bS^1$ with disjoint closures having the following properties (compare also Figure \ref{fig.3}):
\begin{enumerate}
    \item The extreme point of $\overline J$ are fixed points $y_0$ and $y_1$ of $f_0$ and $f_1$,  respectively. Here we are assuming that $y_0\neq y_1$.
    \item There exists an open arc $J^\ast\subset J$ such that $f_0(J^\ast),f_1(J^\ast)\subset \bS^1\setminus \overline{J}$.
    \item The arc $I$ is (forward) invariant, that is, $f_0(I),f_1(I)\subset I$.
    \item For every $x\in \bS^1\setminus \overline{J}$ there exists $n\geq 0$ such that $f_{\xi}^n(x)\in I$ for all $\xi\in\Sigma_2^{+}$.
    \item There is $r\in(0,1)$ so that
\[
	d\left(f_{\xi}^n(x),f_{\xi}^n(y) \right)\leq r^n d(x,y)
	\quad\text{ for every $n\in\bN$ and $x,y\in I$}.
\]
   \item $d(f_i(x),f_i(y))\ge d(x,y)$ for every $x,y\in J\cap f_0^{-1}(J)\cap f_1^{-1}(J)$ and $i=0,1$.
   \item Every $f_i$ is Lipschitz: there is $c>1$ so that
\[
	d(f_i(x),f_i(y))\leq c d(x,y)\quad
	\text{ for every $x,y\in\bS^1$ and $i=0,1$}.
\]
\end{enumerate}
By (6), for the IFS $\cF=\{f_0,f_1\}$ the triple $(\cF,\bp,d)$ is not $\varepsilon$-LCA. An appropriately chosen example also fails to be NEA (just choose $f_0,f_1$ being expanding in $J$). Since $\overline J \subsetneq \mathbb{S}^1$ and $f^{-1}_0\cup f^{-1}_1\left(\overline J\right)\subset \overline J$ we have that $\cF$ fails to be backward minimal, hence methods from \cite{KatSten} do not apply immediately. Below we prove the following.

\begin{lemma}\label{lema1ex2}
The triple $(\cF,\bp,d)$ is proximal, S, SA, and LECA.
\end{lemma}

We will construct $D$ a metric that will be  equivalent to $d$, for which $(\cF,\bp,d)$ is NEA, LECA, and ESCA. This will a metric $\rho$  equivalent to $D$ (and hence $d$), such that $(\cF,\bp,\rho)$ is CA.

By (3) and (5), $\cF$ induces a contracting IFS on $I$. Together with (4),  every $x\not\in\overline J$ eventually enters and remains in $I$. On the other hand, it follows from (2) that 
\[
	f_0^{-1}(\overline J),f_1^{-1}(\overline J)\subset \overline J
\]
and there is a set of points in $X\subset\overline J$ and for every $x\in X$ some sequence $\xi=\xi(x)$ such that $f_\xi^n(x)\in X$ for every $n\in\bN$, though other forward iterates under the IFS $\cF$ eventually leave $\overline J$.  Though, as counterpart and first preliminary result we show that for every $x$
\[
	\{\xi\in\Sigma_2^{+}\colon
	f^n_{\xi}(x)\in J \quad \text{for all}\quad n\in\mathbb{N}\}
\] 
has measure zero. For $k\geq 1$ and $x\in\bS^1$, define
\[\begin{split}
	\Gamma_x^0
	&\eqdef \begin{cases}	\Sigma^{+}_2&\text{ if }x\in I,\\
						\emptyset&\text{ otherwise,}
					\end{cases}	\\
     \Gamma_{x}^k
     &\eqdef \{\xi\in\Sigma^{+}_2\colon
     	f_{\xi}^k(x)\in I, f^{k-1}_{\xi}(x)\notin I\},
	\quad
	\Gamma_x
	\eqdef\bigcup_{k\geq 0}\Gamma_x^k.
\end{split}\]

\begin{lemma}\label{n0exists}
	There exists $N\in\bN$ such that 
\begin{align}\label{eqclaimlem1}
	\mu\Big(\bigcup_{k= 0}^{N+1}\Gamma_x^k \Big)>0
    	\quad\text{for every $x\in\bS^1$}.
\end{align}
and for all $m\geq 1$
\begin{align}\label{eqclaimlem2}
	\mu\big(\Gamma_x^{N+m+1}\big)
	\leq \mu\Big(\Sigma_2^{+}\setminus
		\bigcup_{k= 0}^{N+m}\Gamma_x^k \Big)
	\leq p^m.
\end{align}
Moreover, for every $x\in\bS^1$ it holds $\mu\left(\Gamma_x\right)=1$.
\end{lemma}

\begin{proof}
Fix $y\in J^\ast$. 
By (4), it holds $f_0(y),f_1(y)\notin\overline{J}$. Let  $K,L\subset\bS^1\setminus J$ be the open arcs with extremes $f_0(y)$ and $y_1$ and $f_1(y)$ and $y_0$, respectively. By (4), there are $k=k(K)\in\bN$ and $\ell=\ell(L)\in\bN$ so that
\[
	f_{\xi}^{k}(K),f_{\xi}^{\ell}(L)
	\in I
    	\quad\text{for every $\xi \in\Sigma_2^{+}$}.
\]	
Let $N\eqdef k+\ell$. By (3), for all $\xi\in\Sigma_2^{+}$
\[
	f_{\xi}^{k+\ell}(K),f_{\xi}^{k+\ell}(L)
	\in I
    	\quad\text{for every $\xi \in\Sigma_2^{+}$}.
\]	
Let $W\subset\bS^1\setminus J$ be the closed arc  with extremes $f_0(y)$ and $f_1(y)$. As $f_0$ and $f_1$ preserve orientation, $f_{\xi}^n$ preserves orientation. Hence, $f_{\xi}^n(W)\subset I$ for all $\xi\in\Sigma_2^{+}$. 

Now, let us prove \eqref{eqclaimlem1} and \eqref{eqclaimlem2}. Fix $x\in\bS^1$. 

\medskip\noindent\textbf{Case $x\in W$.} 
As $f_{\xi}^{k+\ell}(x)\in I$ for every $\xi\in\Sigma_2^{+}$, \eqref{eqclaimlem1} and \eqref{eqclaimlem2} are immediate.

\medskip\noindent\textbf{Case $x\not\in W$.} Let us construct a sequence $\xi\in\Sigma_2^{+}$ such that for all $m\geq 1$
\begin{align}\label{seqxi7.6}
    [\xi_1]\cup[\eta_1,\xi_2]\cup\cdots\cup[\eta_1,\ldots,\eta_{m-1},\xi_{m}]
    \subset \bigcup_{k=0}^{m+N}\Gamma_x^k,
\end{align}
where $\eta_i\in\{0,1\}\backslash\{\xi_i\}$.
Since $x$ is either between $y$ and $f_0(y)$ or between $y$ and $f_1(y)$, there exists $\xi_1\in\{0,1\}$ such that $x$ is between $y$ and $f_{\xi_1}(y)$. As $f_{\xi_1}$ preserves orientation, $f_{\xi_1}(x)\in W$. So that 
\[	
	[\xi_1]\subset\bigcup_{k=0}^{N+1}\Gamma_x^k,
\]	
that is, \eqref{seqxi7.6} holds for $m=1$. Let $\eta_1\in\{0,1\}\backslash\{\xi_1\}$. If $f_{\eta_1}(x)\in W$ then $\bigcup_{k=0}^{N+1}\Gamma_x^k=\Sigma_2^{+}$ which implies \eqref{seqxi7.6} for all $m\geq 2$ and $\xi_m\in\{0,1\}$. If $f_{\eta_1}(x)\notin W$ then $f_{\eta_1}(x)$ is either between $y$ and $f_0(y)$ or between $y$ and $f_1(y)$ so there exists $\xi_2\in\{0,1\}$ such that $f_{\eta_1}(x)$ is between $y$ and $f_{\xi_2}(y)$. As $f_{\xi_2}$ preserves orientation, $f_{\xi_2}(f_{\eta_1}(x))\in W$. So that 
\[
	[\xi_1]\cup[\eta_1,\xi_2]\subset\bigcup_{k=0}^{N+2}\Gamma_x^k,
\]	
that is, \eqref{seqxi7.6} holds for $m=2$. Let $\eta_2\in\{0,1\}\backslash\{\xi_2\}$. If $f_{\eta_2}(f_{\eta_1}(x))\in W$ then $\bigcup_{k=0}^{N+2}\Gamma_x^k =\Sigma_2^{+}$ which implies \eqref{seqxi7.6} for all  $m\geq 3$ and $\xi_m\in\{0,1\}$. If $f_{\eta_2}(f_{\eta_1}(x))\notin W$ then $f_{\eta_2}(f_{\eta_1}(x))$ is either between $y$ and $f_0(y)$ or between $y$ and $f_1(y)$ so that there exists $\xi_3\in\{0,1\}$ such that $f_{\eta_2}(f_{\eta_1}(x))$ is between $y$ and $f_{\xi_3}(y)$. Continuing this process inductively on $m$ we conclude \eqref{seqxi7.6}.

Therefore
\[
	[\xi_1]\subset\bigcup_{k= 0}^{N+1}\Gamma_x^k
	\quad\text{ and hence }\quad
	\mu\Big(\bigcup_{k= 0}^{N+1}\Gamma_x^k \Big)
	\geq 1-p>0,
\]	
and so \eqref{eqclaimlem1} holds. Moreover 
\[
	\Gamma_x^{N+m+1}
	\subset \Sigma_2^{+}\setminus\bigcup_{k= 0}^{m+N}\Gamma_x^k\subset [\eta_1,\ldots,\eta_m]
\]	
so that $\mu(\Gamma_x^{N+m+1})\leq p^m$.
This proves the lemma.
\end{proof}

\begin{proof}[Proof of Lemma \ref{lema1ex2}]
Given any $x,y\in\mathbb{S}^1$,  let $\xi\in\Gamma_x\cap\Gamma_y$ and choose $k\in\bN$  such that $f^k_{\xi}(x),f^k_{\xi}(y)\in I$. By  (4),  it holds
\[
	\lim_{n\to\infty}d(f_{\xi}^{n}(x),f_{\xi}^{n}(y))
	= \lim_{n\to\infty}d(f_{\xi}^{k+n}(x),f_{\xi}^{k+n}(y))
	=0.
\]	
As by Lemma \ref{n0exists}, $\mu(\Gamma_x\cap\Gamma_y)=1$ holds, it follows that $(\mathcal{F},\bp,d)$ is S. By Lemma \ref{SAimpliesprox}, it is SA and proximal. By Lemma \ref{LECAimpl}, it is LECA.
\end{proof}

Fix $\ell\geq1$ and pick $\alpha\in(0,1)$ such that 
\begin{equation}\label{el mmmmm}
	r^\ell c<1
	\quad\text{ and }\quad
	c^{\alpha}
	<c^{(\ell+1)\alpha}
	<\frac{1}{p}.
\end{equation}
The choice of these numbers will be apparent in the proof of Lemma \ref{pqp1}.
Consider the metric $d^{\alpha}$ on $\bS^1$ and define the metric $\hat D$ on $\bS^1$ by 
\[
	\hat D(x,y)
	\eqdef \mathbb{E}\left( \sup_{n\geq 0}Z_{n,d^{\alpha}}^{x,y}\right).
\]	
Since $d$ and $d^{\alpha}$ are  equivalent, $\hat D$ and $d$ are  equivalent by Lemma \ref{lemaABC1}.

\begin{lemma}\label{topequivex}
	It holds
\[
	d^\alpha
	\le \hat D
	\le Cd^\alpha
        \quad\text{ where }\quad
	C\eqdef c^{(N+1)\alpha}
		\Big(1+ 2\sum_{k\geq 1}\left(c^{\alpha}p\right)^k\Big).
\]	
\end{lemma} 

\begin{proof}
Clearly $d^{\alpha}\le \hat D$. Let us show the other inequality.
For $k\geq 1$ and $x,y\in\bS^1$ define
\[
	\Gamma_{x,y}^k
	\eqdef \big\{\xi\in\Sigma^{+}_2\colon f_{\xi}^k(x)\in I, f_{\xi}^k(y)\in I\big\}
		\cap
		\big\{\xi\in\Sigma^{+}_2\colon f_{\xi}^{k-1}(x)\notin I \mbox{ or } 
		f_{\xi}^{k-1}(y)\notin I\big\}.
\]	
If $x,y\in I$, then let $\Gamma_{x,y}^0\eqdef \Sigma^{+}_2$. If $x\notin I$ or $y\notin I$, then let $\Gamma_x^0\eqdef\emptyset$. Note that $\{\Gamma_{x,y}^k\}_k$ is a family of pairwise disjoint sets and $\Gamma_{x,y}^k\subset \Gamma_{x}^k\cup \Gamma_{y}^k $. Moreover, it is immediate from the definition that 
\[
	\Gamma_x\cap\Gamma_y
	\subset \Gamma_{x,y}
	\eqdef \bigcup_{n\ge0}\Gamma^n_{x,y}.
\]
Hence, together with Lemma \ref{n0exists}, it follows
\[
	1 
	= \mu(\Gamma_x\cap\Gamma_y)
	= \mu(\Gamma_{x,y})
	= \sum_{n\geq 0}\mu\left(\Gamma_{x,y}^n\right).
\]

Fix $x,y\in\bS^1$. For every $\xi\in \Gamma_{x,y}^k$, by (7) and (5), we have that
\[
	\sup_{n\geq 0}Z_{n,d^{\alpha}}^{x,y}(\xi)
	= \sup_{n\geq 0}d^{\alpha}(f_{\xi}^{n}(x),f_{\xi}^{n}(y))
	< c^{k\alpha }d^{\alpha}(x,y).
\]		
It follows that
\begin{align*}
    \hat D(x,y)
    &=\mathbb{E}\Big( \sup_{n\geq 0}Z_{n,d^{\alpha}}^{x,y}\Big)
    =\sum_{k\geq 0}\mathbb{E}
    	\Big(\sup_{n\geq 0}Z_{n,d^{\alpha}}^{x,y}\mathbbm{1}_{\Gamma_{x,y}^k}\Big)\\
    &\leq c^{(N+1)\alpha}d^{\alpha}(x,y)
    	\mu\Big(\Gamma_{x,y}^0\cup\ldots\cup \Gamma_{x,y}^{N+1}\Big)
	+d^{\alpha}(x,y)\sum_{k\geq N+2}c^{k\alpha}\mu\Big( \Gamma_{x,y}^k\Big)\\
    &\leq c^{(N+1)\alpha}d^{\alpha}(x,y)\Big(1+ \sum_{k\geq N+2}c^{(k-N-1)\alpha}\big(\mu( \Gamma_{x}^k)
    					+\mu( \Gamma_{y}^k)\big)\Big)\\
    &\leq c^{(N+1)\alpha}d^{\alpha}(x,y)
    	\Big(1+2 \sum_{k\geq N+2}c^{(k-N-1)\alpha}p^{k-N-1}
    \Big)\\
    &\leq Cd^{\alpha}(x,y).
\end{align*}
As $x,y$ were arbitrary, this finishes the proof.
\end{proof}

\begin{lemma}\label{pqp1}
The triple $(\cF,\bp,\hat D)$ is SA, NEA, LECA and ESCA.
\end{lemma}

\begin{proof}
By Lemmas \ref{topequivex} and \ref{sec3lem2}, $(\cF,\bp,\hat D)$ is S. Hence, by which Lemma \ref{sec3lem1}, $(\cF,\bp,\hat D)$ is SA, proving the first assertion.

To show NEA and LECA, check that
\begin{align*}
    \mathbb{E}\left(Z_{1,\hat D}^{x,y}\right)
    &=p\,\hat D(f_0(x),f_0(y))+(1-p)\,\hat D(f_1(x),f_1(y))\\
    &=p\,\mathbb{E}\left(\sup_{n\geq 0}Z_{n,d^{\alpha}}^{f_0(x),f_0(y)}\right)	
    	+(1-p)\,\mathbb{E}\left(\sup_{n\geq 0}Z_{n,d^{\alpha}}^{f_1(x),f_1(y)}\right)\\
    &= p\,\mathbb{E}\left(\mathbbm{1}_{[0]}
    		\sup_{n\geq 0}Z_{n+1,d^{\alpha}}^{x,y}\right)
    	+(1-p)\,\mathbb{E}\left(\mathbbm{1}_{[1]}
		\sup_{n\geq 0}Z_{n+1,d^{\alpha}}^{x,y}\right)
	=\mathbb{E}\left(\sup_{n\geq 1}Z_{n,d^{\alpha}}^{x,y}\right).
\end{align*}
This implies $\mathbb{E}(Z_{1,\hat D}^{x,y})\le \hat D(x,y)$, that is, $(\cF,\bp,\hat D)$ is NEA.
Hence, as $(\cF,\bp,\hat D)$ is NEA. As $(\cF,\bp,\hat D)$ is NEA and SA, Lemma \ref{lem:2}  property LECA follows.

Now, let us prove that $(\cF,\bp,\hat D)$ is  ESCA. 
Let $N$ be as in Lemma \ref{n0exists}, $\ell$ as in \eqref{el mmmmm}, and $C$ is as in Lemma \ref{topequivex}. 
Recalling that $r,p\in(0,1)$, we can fix $n\in\bN$ sufficiently large such that 
\begin{equation}\label{constantemedia}
    (r^\ell c)^{n\alpha}+c^{n(\ell+1)\alpha}p^{n-N}
    <\frac{1}{2C}.
\end{equation}
 Recall that  $\bigcup_{k= 0}^{N+1} \Gamma_x^{k}\neq\emptyset$ for every  $x\in\bS^1$. Note that $\bigcup_{k=0}^{N+1}\Gamma_x^k$ is covered by cylinders of length $N$ and, in particular, $\{f_\xi^N\colon \xi\in \bigcup_{k=0}^{N+1}\Gamma_x^k\}$ is a collection of at most $2^N$ homeomorphisms.
Hence, given $x\in\bS^1$, the set
\[
	V_x
	\eqdef \bigcap_{\xi\in\bigcup_{k= 0}^N 
		\Gamma_x^{k}}\left(f_{\xi}^N\right)^{-1}(I),
\]		
is, as an intersection of finitely many open intervals containing $x$, an open neighborhood of $x$. Moreover, for every $y,z\in V_x$ it holds
\begin{align*}
    \bE&\left(Z_{(\ell+1)n,d^{\alpha}}^{y,z}\right)
    =\bE\left(
    		\mathbbm{1}_{\bigcup_{k=0}^{n} \Gamma_x^k}
		Z_{(\ell+1)n,d^{\alpha}}^{y,z}\right)
	+\bE\left(
		\mathbbm{1}_{\Sigma^{+}_2
			\setminus \bigcup_{k= 0}^{n} \Gamma_x^{k}}
		Z_{(\ell+1)n,d^{\alpha}}^{y,z}\right)\\
&=\bE\left(
    		\mathbbm{1}_{\bigcup_{k=0}^{n} \Gamma_x^k}
		d^\alpha(f_{\sigma^n(\xi)}^{\ell n}(f_\xi^n(y)),
				f_{\sigma^n(\xi)}^{\ell n}(f_\xi^n(z))\right)
	+\bE\left(
		\mathbbm{1}_{\Sigma^{+}_2
			\setminus \bigcup_{k= 0}^{n} \Gamma_x^{k}}
		Z_{(\ell+1)n,d^{\alpha}}^{y,z}\right)\\
	{\tiny{\text{(by (7))}}}			
    &\leq c^{n\alpha}r^{\ell n\alpha}
    	d^{\alpha}(y,z)\mu\left( \bigcup_{k= 0}^{n} \Gamma_x^{k} \right)	
	+c^{(\ell+1)n\alpha}d^{\alpha}(y,z)
		\mu\left(\Sigma^{+}_2\setminus  \bigcup_{k= 0}^{n} \Gamma_x^k\right)\\
    &\leq \left((r^\ell c)^{n\alpha}+c^{(\ell+1)n\alpha}p^{n-N}\right)d^{\alpha}(y,z)\\
    &\leq \frac{1}{2C} d^{\alpha}(y,z).
\end{align*}
By Lemma \ref{topequivex}, for every $y,z\in V_x$
\[
	\frac{\bE\Big(Z_{(m+1)N,\hat D}^{y,z}\Big)}{\hat D(y,z)}
	\leq C\frac{\bE\Big(Z_{(m+1)N,d^\alpha}^{y,z}\Big)}
			{d^{\alpha}(y,z)}\leq \frac{1}{2}, 
\]			
which implies that
\[
	\sup_{(y,z)\in V_x\times V_x}\frac{\bE\Big(Z_{(m+1)N,\hat D}^{y,z}\Big)}{\hat D(y,z)}
	\leq \frac{1}{2}<1.
\]	
Since $x$ is arbitrary, this proves ESCA, and completes the proof.
\end{proof}

\begin{remark}
The constant $\frac{1}{2}$ in equation \eqref{constantemedia} is insignificant, in fact, we can change $\frac{1}{2}$ for any $t\in(0,1)$ and get the same result.
\end{remark}

By Theorem \ref{pro:1},  Lemma \ref{pqp1} allows us to conclude that there exists a metric $D$ strongly equivalent to $\hat D$ such that $(\cF,\bp,D)$ is CA. Hence, $D$ is strongly equivalent to $d^{\alpha}$. 

\begin{remark}
Note that in the proof of Lemma \ref{pqp1} we show that $(\cF,\bp,d^{\alpha})$ is ESCA. The triple $(\cF,\bp,d^{\alpha})$ is an example of a system that is ESCA such that the $\ell$ required in the definition is not the constant $1$. 
\end{remark}

\bibliographystyle{alpha}

\end{document}